\documentclass[a4paper,12pt,reqno]{amsart}
\usepackage{amssymb,amsmath,amsfonts, amscd}
\usepackage{mathrsfs,latexsym}
\usepackage[all,ps,cmtip]{xy}
\usepackage[usenames,dvipsnames,svgnames,table]{xcolor}
\title{On projective $3$-folds of general type with  $p_g=2$ }
\author{Meng Chen, Yong Hu, Matteo Penegini}

\address{\rm School of Mathematical Sciences, Fudan University, Shanghai 200433, China}
\email{mchen@fudan.edu.cn}

\address{\rm School of Mathematics, Korea Institute for Advanced Study, 85 Hoegiro, Dongdaemun-gu, Seoul 02455, South Korea }
\email{yonghu11@kias.re.kr}

\address{\rm Universit\`a degli Studi di Genova, DIMA Dipartimento di Matematica, I-16146 Genova, Italy }
\email{penegini@dima.unige.it}

\thanks{The first author was supported by National Natural Science Foundation of China (\#11571076, \#11731004) and Program of Shanghai Subject Chief Scientist (\#16XD1400400). The second author was supported by National Researcher Program of National Research Foundation of Korea (Grant No.~2010-0020413). The third author was supported by PRIN 2015 ``Geometry of Algebraic Varieties'' and by GNSAGA of INdAM}

\dedicatory{Dedicated to professor Miles Reid on his seventieth birthday }
\newcommand{\bC}{{\mathbb C}}
\newcommand{\bQ}{{\mathbb Q}}
\newcommand{\bP}{{\mathbb P}}
\newcommand{\roundup}[1]{\lceil{#1}\rceil}
\newcommand{\rounddown}[1]{\lfloor{#1}\rfloor}

\newcommand\lrw{\longrightarrow}

\newcommand\OO{{\mathcal{O}}}
\newcommand\OX{{\mathcal{O}_X}}

\newcommand{\lsgeq}{\succcurlyeq}
\newcommand{\lsleq}{\preccurlyeq}

\newcommand{\Mov}{\text{Mov}}

\newcommand{\simQ}{\sim_{\mathbb{Q}}}

\newtheorem{thm}{Theorem}[section]
\newtheorem{lem}[thm]{Lemma}

\newtheorem{prop}[thm]{Proposition}

\theoremstyle{definition}
\newtheorem{defn}[thm]{Definition}

\newtheorem{rem}[thm]{Remark}
\theoremstyle{remark}

\begin{document}

\maketitle

\begin{abstract}  We classify minimal projective $3$-folds of general type with  $p_g=2$ by studying the birationality of their $6$-canonical maps.
\end{abstract}
\maketitle

%%%%%%%%%
\pagestyle{myheadings}
\markboth{\hfill M. Chen, Y. Hu and M. Penegini\hfill}{\hfill On projective $3$-folds of general type with $p_g=2$ \hfill}
\numberwithin{equation}{section}
%%%%%%%%%%%%
%\tableofcontents

\section{\bf Introduction}

The study of pluricanonical maps is a fundamental aspect of birational geometry. Let  $\varphi_m$ be the m-canonical map of a projective variety $X$. It is known, by Hacon-McKernan \cite{H-M}, Takayama \cite{Ta} and Tsuji \cite{Tsuji}, that there exists a constant $r_n$ (for any integer $n>0$) such that the pluricanonical map $\varphi_m$ is birational onto its image for all $m\geq r_n$ and for all smooth projective $n$-folds of general type.  Despite the great efforts of several authors, $r_n$ is not explicily  given except for $n\leq 3$. By now it is a classical result for curves and surfaces that $r_1=3$ and $r_2=5$ (see \cite{Bom}). In addition, very recently Chen  and the first author proved the bound  $r_3\leq 57$ (see \cite{EXP1,EXP2, EXP3,Delta18}).

Provided that the property we are studying is birationally invariant, the $3$-dimensional MMP allows us to work with any minimal model $X$ ($\bQ$-factorial with at worst terminal singularities) of  a nonsingular projective $3$-fold of general type.  The aim of this paper is to give a more precise bound for $r_3$ adding some extra information on the nature of $X$.  We have already studied the birational geomentry of projective $3$-folds of general type with  geometric genus $p_g=1$ and $3$ in \cite{CHP}. Here, we shall assume that the geometric genus $p_g$ of $X$ is equal to 2. Under this hypothesis, we are able to classify minimal projective $3$-folds of general type with  $p_g=2$ by studying only the birationality of their $6$-canonical maps. 

We need to introduce some terminology in order to state  main result of this paper: Theorem \ref{thm:main p_g2} which was announced in \cite{CHP}.

By Chen-Chen's series of works in \cite{EXP1,EXP2,EXP3}, there exists a positive integer $m_0\leq 18$ such that $P_{m_0}(X)=h^0(X, m_0K_X)\geq 2$. Hence it is possible to investigate the birational geometry of $X$ by studying the behavior of the $m_0$-canonical map $\varphi_{m_0, X}$.  This strategy proves to be very effective.

\begin{defn} Let $W$ be a $\bQ$-factorial normal projective variety of dimension $n$. Assume that the two maps $\tau:W\dashrightarrow W'$ and $g:W'\lrw S$ satisfy the following properties:
	\begin{itemize}
		\item[(1)] $W'$ is a nonsingular projective variety and $S$ is normal projective of dimension $s<n$;
		\item[(2)] $\tau$ is a dominant birational map and $g$ is a fibration.
	\end{itemize}
	Then we say that the set
	$$\mathcal{F}=\{\hat{F}\subset W|\hat{F}=\tau_*^{-1}(F), F\ \text{is a  fiber of}\ g\}$$
	{\it forms an $(n-s)$-fold class} of $W$, where $\tau_*^{-1}(\cdot)$ denotes the strict transform. In particular, if $n-s=1$ ($=2$), we call it a \emph{curve class} (a \emph{surface class}). The number $(K_W^{n-s}\cdot \tau_*^{-1}(F))$ ($F$ a general fiber of $g$) is called {\it the canonical degree of $\mathcal{F}$}.  Such degree is also denoted as ``$\deg_{c}(\mathcal{F})$''.
\end{defn}

Especially, when $\varphi_{m_0,X}$ is of fiber type (i.e. $\dim \overline{\varphi_{m_0,X}(X)}<\dim X$), the induced fibration (obtained by taking the Stein factorization of $\varphi_{m_0,X}$) automatically forms either a curve class $\mathcal{C}$ or a surface class $\mathcal{S}$ of $X$.  We also say that $X$ is $m_0$-canonically fibred by a curve class  $\mathcal{C}$ (or a surface class $\mathcal{S}$). Note that in our case $m_0=1$, so we can simply say canonically fibred. We can now state the main theorem.

\begin{thm}\label{thm:main p_g2}
	{\em Let $X$ be a minimal projective $3$-fold of general type with $p_g(X)= 2$. Then one of the following statements is true:
		\begin{itemize}
			\item[(1)] $\varphi_{6,X}$ is birational onto its image;
			\item[(2)] $X$ is canonically fibered by a $(2,3)$-surface class of canonical degree $\frac{1}{2}$, in which case $\varphi_{6,X}$ is non-birational;
			\item[(3)] $X$ is canonically fibered by a $(1,2)$-surface class
			(denote by $\mathcal{C}$ the genus $2$ curve class which is naturally induced from $\mathcal{S}$) and one of the following holds:
			\begin{enumerate}
				\item[(i)] $\deg_c(\mathcal{C})=\frac{2}{3}$;
				\item[(ii)] $\deg_c(\mathcal{C})=\frac{4}{5}$;
				\item[(iii)] $P_2(X)=5$, $\deg_c(\mathcal{C})=1$ and $\deg_c(\mathcal{S})=\frac{1}{2}$;
				\item[(iv)] $P_4(X)=14$, $\deg_c(\mathcal{C})=1$ and $\deg_c(\mathcal{S})=\frac{1}{2}$.
			\end{enumerate}
			In this case $\varphi_{6,X}$ is non-birational. 
			\item[(4)] There is an explicit finite set ${\mathbb S}_2$ such that $X$ is canonically fibered by a $(1,2)$-surface class and ${\mathbb B}(X)\in {\mathbb S}_2$, in which case $\varphi_{6,X}$ is non-birational.  (see  Subsection \ref{basket} for the definition of ${\mathbb B}(X)$, the weighted basket of $X$) 	\end{itemize}}
\end{thm}

\begin{rem} The existence of threefolds described in Theorem \ref{thm:main p_g2}  (4) are provided by the following examples.   Denote by $X_d$ a general weighted hypersurface of degree $d$ in the sense of Fletcher (see \cite{F00}).
\begin{enumerate} 
\item The 3-fold $X_{16}\subset \bP(1,1,2,3,8)$ has $K^3=\frac{1}{3}$, $p_g=2$ and $\varphi_7$ is non-birational;
\item The 3-fold $X_{14}\subset \bP(1,1,2,2,7)$ has $K^3=\frac{1}{2}$, $p_g=2$ and $\varphi_6$ is non-birational.
	\end{enumerate}
	Moreover,
\begin{enumerate} 	
\item We do not know whether any threefold with properties described in Theorem \ref{thm:main p_g2}  (2) might exist, nor if all those encoded by  ${\mathbb S}_2$ exist (most likely not). 
\item A complete list of the 263 elements of the set  ${\mathbb S}_2$  can be found at the following webpage.
	\begin{center}
		\verb|http://www.dima.unige.it/~penegini/publ.html|
	\end{center}   
	\end{enumerate}
\end{rem}

The plan of the paper is the following:

In Section 2, we describe the set up of the work. We recall some key theorems for the study of the pluricanonical maps for 3-folds of general type  and some necessary inequalities in a general frame work. Moreover we introduce the notion of weighted basket. 

Section 3 contains the core technical theorems of the paper,  which will be effectively used to do the classification. These theorems concern 3-folds with $p_g \geq 2$ and canonically fibered by a $(1,2)$-surface class. 

Theorem \ref{thm:main p_g2} is proved in several steps  in Section 4, which is the longest section of the paper. Subsection 4.1 takes care of Theorem \ref{thm:main p_g2} cases (1) and (2).   Theorem \ref{thm:main p_g2} cases (3)(i) and (3)(ii) are proved in Subsection 4.2.  Most of Section 4  (Subsection 4.3 and 4.4) is then devoted to constructing effective numerical constraints on $P_2(X)$, $P_3(X)$, $P_4(X)$, $P_5(X)$ and $P_6(X)$. This is done by  repeatedly applying  the theorems of Section 3 in a rigorous case by case analysis. These constrains on the plurigenera will be used to produce (by computer aided computation) the set  ${\mathbb S}_2$ that proves Theorem \ref{thm:main p_g2} case (4) (see Subsection 4.5).  Finally, in Subsection 4.3, Theorem \ref{thm:main p_g2} cases (3)(iii) and (3)(iv) (See Propositions \ref{prop:P2}, \ref{prop:P4}) are proved. This section provides also more details and insights on the computations done in \cite{CHP}, where the tedious calculations are omitted (see Proposition X \cite{CHP}). 

\medskip

\textbf{Notation and conventions.} 
We work over the field $\bC$ of complex numbers.  A minimal threefold of general type $X$ is a $\bQ$-factorial 3-fold  with at worst terminal singularities such that the canonical divisor $K_X$ is a nef and big $\bQ$-Cartier divisor.  Moreover, let $\omega_X = \mathcal{O}_X(K_X)$ be the canonical sheaf. Throughout the paper we use the following symbols. 
\begin{itemize}
\item[$\diamond$] ``$\sim$'' denotes linear equivalence or ${\mathbb Q}$-linear equivalence when specified ``$\sim_{\bQ}$''; 
\item[$\diamond$] ``$\equiv$'' denotes numerical equivalence; 
\item[$\diamond$] ``$|M_1|\lsgeq |M_2|$'' (or, equivalently,  
``$|M_2| \lsleq |M_1|$'') means, for linear systems $|M_1|$ and $|M_2|$ on a variety,  
$$|M_1|\supseteq|M_2|+\text{(fixed effective divisor)}.$$
\end{itemize}

\section{\bf Preliminaries}

\subsection{Set up}\label{setup}

Let $X$ be a minimal projective 3-fold of general type and we assume that  $p_g(X)=h^0(X, \OX(K_X))\geq 2$. So we may consider the canonical map $\varphi_1: X\dashrightarrow \bP^{p_g(X)-1}$, which is a non-constant rational map.

{}From the very beginning we fix an effective Weil divisor $K_1\sim K_X$. Take successive blow-ups $\pi: X'\rightarrow X$, which exists by Hironaka's big theorem, such that:
\smallskip

(i) $X'$ is nonsingular and projective;

(ii) the moving part of $|K_{X'}|$ is base point free;

(iii) the union of supports of both $\pi^*(K_1)$ and exceptional divisors of $\pi$ is simple normal crossing.
\smallskip

Denote by $\tilde{g}$ the composition $\varphi_{1}\circ\pi$. So $\tilde{g}:
X'\rightarrow \Sigma\subseteq{\mathbb P}^{p_g(X)-1}$ is a non-constant morphism by the above assumption.
Let $X'\overset{f}\rightarrow \Gamma\overset{s}\rightarrow \Sigma$ be
the Stein factorization of $\tilde{g}$. We get the following commutative diagram:
\medskip

\begin{picture}(50,80) \put(100,0){$X$} \put(100,60){$X'$}
\put(170,0){$\Sigma$} \put(170,60){$\Gamma$} \put(115,65){\vector(1,0){53}}
\put(106,55){\vector(0,-1){41}} \put(175,55){\vector(0,-1){43}}
\put(114,58){\vector(1,-1){49}} \multiput(112,2.6)(5,0){11}{-}
\put(162,5){\vector(1,0){4}} \put(133,70){$f$} \put(180,30){$s$}
\put(92,30){$\pi$} \put(132,-6){$\varphi_1$}\put(136,40){$\tilde{g}$}
\end{picture}
\bigskip

\noindent We may write $K_{X'}=\pi^*(K_X)+E_{\pi}$, where $E_{\pi}$ is an effective ${\bQ}$-divisor which is a sum of distinct exceptional divisors with positive rational coefficients.
By definition, for any positive integer $m$, we have
$\roundup{m\pi^*(K_X)}\leq mK_{X'}$. Set $|M|=\text{Mov}|K_{X'}|$.
Since
$$h^0(X', M)=h^0(\OX(K_{X'}))$$
and $X$ has at worst terminal singularities,
we may also write $$\pi^*(K_X)\sim_{\mathbb Q}
M+E',$$ where $E'$ is another effective ${\mathbb Q}$-divisor. Set
$$d_1=\dim\overline{\varphi_1(X)}=\dim(\Gamma).$$ Clearly one has $1\leq d_1\leq 3$.

If $d_1=2$, a general fiber of $f$ is a smooth
projective curve of genus $\geq 2$. We say that $X$ is {\it canonically fibred by curves}.

If $d_1=1$, a general fiber $F$ of $f$ is a smooth
projective surface of general type. We say that $X$ is {\it
canonically fibred by surfaces} with invariants $(c_1^2(F_0), p_g(F_0)),$ where $F_0$ is the minimal model of $F$ via the contraction morphism $\sigma: F\rightarrow F_0$. We may write $M\equiv a F$ where $a=\deg f_*\OO_{X'}(M)$. %Denote $b=g(\Gamma)$.

Just to fix the notion, {\it a generic irreducible element $S$ of} $|M|$ means either a general member of $|M|$ in the case of $d_1\geq 2$ or, otherwise, a general fiber $F$ of $f$.

For any positive integer $m$, $|M_m|$ denotes the moving part of $|mK_{X'}|$.
Let $S_m$ be a general member of $|M_m|$ whenever $m>1$.

Set
$$\zeta=\begin{cases}
1, &\text{if}\ d_1\geq 2;\\
a, & \text{if}\ d_1=1.
\end{cases}$$
Naturally one has $\pi^*(K_X)\simQ \zeta S+E'.$ In practice we need such a real number
$\mu=\mu(S)$ which is defined to be the supremum of those rational numbers $\mu'$ satisfying the following property:
\begin{equation} \pi^*(K_X)\simQ \mu' S+E_{S}'\label{muS}\end{equation}
for certain effective $\bQ$-divisor $E_{S}'$. Clearly we have $\mu(S)\geq \zeta$.

\subsection{Convention} For any linear system $|D|$ of positive dimension on a normal projective variety $Z$, we may write
$$|D|=\text{Mov}|D|+\text{Fix}|D|$$ and consider the rational map $\Phi_{|D|}=\Phi_{\text{Mov}|D|}$.  We say that {\it $|D|$ is not composed of a pencil if $\dim \overline{\Phi_{|D|}(Z)}\geq 2$}. A {\it generic irreducible element} of $|D|$ means a general member of $\text{Mov}|D|$ when $|D|$ is not composed of a pencil or, otherwise, an irreducible component in a general member of $\text{Mov}|D|$.
For a nonsingular projective surface $S$ of general type, we say that {\it $S$ is a $(u,v)$
- surface} if $K_{S_0}^2=u$ and $p_g(S_0)=v$ where $S_0$ is the minimal model of $S$.

\subsection{Known inequalities}\label{ss2} Pick a generic irreducible element $S$ of $|M|$. Clearly, $S$ is a nonsingular projective surface of general type. Assume that $|G|$ is a base point free linear system on $S$. Denote by $C$ a generic irreducible element of $|G|$.
Since $\pi^*(K_X)|_S$ is nef and big, there is a rational number $\beta>0$ such that
$\pi^*(K_X)|_S\geq \beta C$. Granted the existence of such $\beta$, we may assume from now on that $\beta=\beta(|G|)$ is the supremum satisfying the above property.

For any integer $m>0$, we define
\begin{eqnarray*}
\xi=\xi(|G|)&=&(\pi^*(K_X)\cdot C)_{X'},\\
\alpha(m)=\alpha_{(|G|)}(m)&=& (m-1-\frac{1}{\mu}-\frac{1}{\beta})\xi,\\
\alpha_0(m)&=&\roundup{\alpha(m)}.
\end{eqnarray*}
When no confusion arises as it is likely in the context, we will simply use the simple notation $\zeta$, $\mu$, $\beta$, $\xi$ and $\alpha(m)$.
According to \cite[Theorem 2.11]{EXP2}, whenever $\alpha(m)>1$, one has
\begin{equation} m\xi\geq \deg(K_C)+\alpha_0(m).\label{kieq1}
\end{equation}
In particular, as $m$ is sufficiently large so that $\alpha(m)>1$, Inequality \eqref{kieq1} implies
\begin{equation}
\xi\geq \frac{\deg(K_C)}{1+\frac{1}{\mu}+\frac{1}{\beta}}. \label{kieq2}
\end{equation}

Moreover, by \cite[Inequality (2.1)]{MA} one has
\begin{equation}\label{kcube}
K^3_X \geq \mu \beta\xi.
\end{equation}

\subsection{Birationality principle} We refer to \cite[2.7]{EXP2} for birationality principle.  Recall the following concept for point separations.

\begin{defn} Let $|L|$ be a moving (without fixed part) linear system on a normal projective variety $Z$. We say that the rational map $\Phi_{|L|}$ {\it distinguishes sub-varieties $W_1, W_2\subset Z$} if, set theoretically, $$\overline{\Phi_{|L|}(W_1)}\nsubseteqq \overline{\Phi_{|L|}(W_2)}\ \text{and}\ \overline{\Phi_{|L|}(W_2)}\nsubseteqq \overline{\Phi_{|L|}(W_1)}.$$ We say that $\Phi_{|L|}$ {\it separates points $P, Q\in Z$} (for $P, Q\not\in \text{Bs}|L|$), if $\Phi_{|L|}(P)\neq \Phi_{|L|}(Q)$.
\end{defn}

We will tacitly and frequently use the following theorem in the context:

\begin{thm}\label{key-birat} (see \cite[Theorem 2.11]{EXP2}) Keep the same setting and assumption as in Subsection \ref{setup} and Subsection \ref{ss2}.  Pick up a generic irreducible element $S$ of $|M|$. For $m>0$, assume that the following conditions are satisfied:
\begin{itemize}
\item[(i)] $|mK_{X'}|$ distinguishes different generic irreducible elements of $|M|$;
\item[(ii)] $|mK_{X'}||_S$ distinguishes different generic irreducible elements of $|G|$;
\item[(iii)] $\alpha(m)>2$.
\end{itemize}
Then $\varphi_{m,X}$ is birational onto its image.
\end{thm}

\subsection{A weak form of extension theorem} Sometimes we use the following theorem which is a special form of Kawamata's extension theorem (see \cite[Theorem A]{KawaE}):

\begin{thm}\label{KaE}  (see \cite[Theorem 2.4]{MZ2016}) Let $Z$ be a nonsingular projective variety on which $D$ is a smooth divisor such that $K_Z+D\simQ A+B$ for an ample $\bQ$-divisor $A$ and an effective $\bQ$-divisor $B$ and that $D$ is not contained in the support of $B$. Then the natural homomorphism
$$H^0(Z, m(K_Z+D))\longrightarrow H^0(D, mK_D)$$
is surjective for all $m>1$.
\end{thm}

In particular, when $Z$ is of general type and $D$ moves in a base point free linear system, the condition of Theorem \ref{KaE} is automatically satisfied.  Taking  $Z=X'$, $D=S$ and modulo a process of taking the limit (so may assuming $\mu$ to be rational), it holds that
$$|n(\mu+1)K_{X'}||_S\lsgeq |n\mu(K_{X'}+S)||_S=|n\mu K_S|$$
for some sufficiently large and divisible integer $n$. Noting that
$$n(\mu+1)\pi^*(K_X)\geq M_{n(\mu+1)}$$
and that $|n(\mu+1)\sigma^*(K_{S_0})|$ is base point free, we have
\begin{equation}
\pi^*(K_X)|_S\geq \frac{\mu}{\mu+1}\sigma^*(K_{S_0})\geq \frac{\zeta}{1+\zeta}\sigma^*(K_{S_0}).\label{cri}
\end{equation}

\subsection{The weighted basket of $X$}\label{basket}

The {\it weighted basket} ($=$ formal basket) $\mathbb{B}(X)$ is defined to be the triple
$\{B_X, P_2(X), \chi(\OO_X)\}$.
We keep all the definitions and symbols in \cite[Sections 2 and 3]{EXP1} such as ``basket'', ``prime packing'',  ``the canonical sequence of a basket'', $\Delta^j(B)$ ($j>0$),  $\sigma$, $\sigma'$, $B^{(n)}$ ($n\geq 0$), $\chi_m(\mathbb{B}(X))$ ($m\geq 2$), $K^3(\mathbb{B}(X))$,  $\sigma_5$, $\varepsilon$, $\varepsilon_n$ ($n\geq 5$) and so on.

As $X$ is of general type, the vanishing theorem and Reid's Riemann-Roch formula \cite{Reid87} (see also front lines in \cite[4.5]{EXP1}) imply that
$$\chi_m(\mathbb{B}(X))=P_m(X)$$
for all $m\geq 2$ and $K^3(\mathbb{B}(X))=K_X^3$.  For any $n\geq 0$, $B^{(n)}$ can be expressed in terms of $\chi(\OO_X)$, $P_2$, $P_3$, $\cdots$, $P_{n+1}$ (see \cite[(3.3)$\sim$(3.14)]{EXP1} for more details), which serves as a considerably powerful tool for our classification.

%%%%%%
%%%%%%%%%
\section{\bf Some technical theorems}
%%%%%
%%%%%%

\subsection{Two restriction maps on canonical class of $(1,2)$-surfaces}\label{3.2}
Within this subsection, we always work under the following assumption:

{\bf ($\pounds$)} {\em Keep the setting in \ref{setup}. Let $m_1>1$ be an integer. Assume that $|M_{m_1}|$ is base point free,  $d_1=1$, $\Gamma\cong \bP^1$ and that $F$ is a $(1,2)$-surface. Take $|G|=\Mov|K_F|$, which is assumed to be base point free. Let $C$ be a generic irreducible element of $|G|$. } 

\begin{defn}\label{twomaps} For any integers $j\geq 0$, define the following restriction maps:
	$$H^0(X', M_{m_1}-jF)\overset{\theta_{m_1,-j}}\lrw H^0(F, M_{m_1}|_F),$$
	$$H^0(F, M_{m_1}|_F-jC)\overset{\psi_{m_1,-j}}\lrw H^0(C, M_{m_1}|_C).$$
	Set $U_{m_1,-j}=\text{Im}(\theta_{m_1,-j})$, $V_{m_1,-j}=\text{Im}(\psi_{m_1, -j})$,
	$u_{m_1, -j}=\dim U_{m_1,-j}$ and $v_{m_1, -j}=\dim V_{m_1, -j}$.
\end{defn}

In this section we prove three technical theorems that relate the numbers $\beta, \xi, \mu$ and $\alpha$ to the linear systems of the definition above. These three theorems will be used systematically in Section 4 together with  \cite[Proposition 3.4, 3.5, 3.6, 3.7]{CHP} while using the setting $m_0=1$.

\begin{thm}\label{k2} Let $X$ be a minimal projective $3$-fold of general type with $p_g(X)\geq 2$. Keep Assumption ($\pounds$).  Let $m_1>1$ be an integer.
	%Suppose that $d_{m_0}=1$, $\Gamma\cong \bP^1$ and that $F$ is a $(1,2)$-surface.
	Suppose that $|S_1|$ is a base point free linear system on $X'$ with $h^0(S_1|_F)\geq 2$ and that, for some integer $j\geq 2$,
	$$M_{m_1}\geq jF+S_1.$$
	Denote by $C_1$ the generic irreducible element of $|S_1|_F|$. Assume that  $|S_1|_F|$ and $|G|$ are not composed of the same pencil.  Set $\tilde{\delta}=(C_1\cdot C)$. Then
	\begin{itemize}
		\item[(i)] when $\tilde{\delta}\le 2j$,
		$$\varphi_{\rounddown{\frac{1}{\xi(|G|)}(2-\frac{\tilde{\delta}}{j})+\frac{m_1}{j}+\frac{1}{\beta}}+2, X}$$  is birational.
		\item[(i)$'$] when $\tilde{\delta}\le j$,
		$$\varphi_{\rounddown{\frac{2}{\xi(|G|)}(1-\frac{\tilde{\delta}}{j})+\frac{2m_1}{j}}+3, X}$$ is birational.
		\item[(ii)] when $\tilde{\delta}>2j$,
		$$\varphi_{\rounddown{\frac{2m_1}{\tilde{\delta}}+\frac{1}{\beta}+\frac{1}{\mu}\cdot(1-\frac{2j}{\tilde{\delta}})}+2,X}$$ is birational.
		\item[(ii)$'$] when $\tilde{\delta}>2j$ and $S_1|_F$ is big,
		$$\varphi_{\roundup{\frac{m_1}{j}+\frac{1}{\beta}}+1, X}$$ is birational.
		\item[(iii)] one has
		$$\pi^*(K_X)|_F\geq \frac{j}{j+m_1} \sigma^*(K_{F_0})+\frac{1}{j+m_1}S_1|_F.$$
		\item[(iv)] For any integer $n>\frac{m_1}{j}+\frac{1}{\beta}$  with
		$$(n-\frac{m_1}{j}-\frac{1}{\beta})\xi+\frac{\tilde{\delta}}{j}>1, $$
		one has
		$$(n+1)\xi\geq \roundup{(n-\frac{m_1}{j}-\frac{1}{\beta})\xi+\frac{\tilde{\delta}}{j}}+2.$$
	\end{itemize}
\end{thm}
\begin{proof}  Let $|G_1|=|S_1|_F|$.  By assumption, $|G_1|$ is also base point free.
	Set
	$$n=\begin{cases}\rounddown{\frac{1}{\xi(|G|)}(2-\frac{\tilde{\delta}}{j})+\frac{m_1}{j}+\frac{1}{\beta}}+1, & \text{when}\  \tilde{\delta}\leq 2j;\\
	
	\rounddown{\frac{2m_1}{\tilde{\delta}}+\frac{1}{\beta}+\frac{1}{\mu}\cdot(1-\frac{2j}{\tilde{\delta}})}+1, & \text{when}\   \tilde{\delta}> 2j.\end{cases}$$
	Write
	$$m_1\pi^*(K_X)\equiv jF+S_1+E_{m_1},$$
	where $E_{m_1}$ is an effective $\bQ$-divisor on $X'$.
	
	By  Kawamata-Viehweg vanishing theorem (\cite{KaV, VV}), we have
	\begin{eqnarray}
	|(n+1)K_{X'}||_F&\lsgeq&|K_{X'}+\roundup{n\pi^*(K_X)-\frac{1}{j}E_{m_1}}||_F
	\notag\\
	&\lsgeq& |K_F+\roundup{n\pi^*(K_X)|_F-\frac{1}{j}E_{m_1}|_F}| \label{ma},
	\end{eqnarray}
	since
	$$n\pi^*(K_X)-\frac{1}{j}E_{m_1}-F\equiv (n-\frac{m_1}{j})\pi^*(K_X)+\frac{1}{j}S_1$$
	is simple normal crossing (by our assumption), nef and big.
	
	Since $p_g(X)>0$, one sees that $|(n+1)K_{X'}|$ distinguishes different general $F$ and
	$|(n+1)K_{X'}||_F$ distinguishes different general $C$. What we need to do is to investigate  the behavior of $|(n+1)K_{X'}||_C$.
	
	%(1).  If $|G_1|$ and $|G|$ are not composed of the same pencil, then
	%$$\xi(m_0, |G|)\geq \frac{1}{m_1}(M_{m_1}|_F\cdot C)\geq \frac{1}{m_1}(C_1\cdot C)\geq \frac{2}{m_1}.$$
	Recall that we have
	\begin{equation}\label{FF1}\frac{1}{\beta}\pi^*(K_X)|_F\equiv C+H_1,\end{equation}
	where $H_1$ is certain effective $\bQ$-divisor.  The vanishing theorem on $F$ gives
	\begin{eqnarray}
	&&|K_F+\roundup{n\pi^*(K_X)|_F-\frac{1}{j}E_{m_1}|_F-H_1}||_C\notag\\
	&\lsgeq& |K_F+\roundup{n\pi^*(K_X)|_F-\frac{1}{j}E_{m_1}|_F-H_1}||_C\notag\\
	&=&|K_C+\tilde{D}_1|\label{m01},
	\end{eqnarray}
	where $\tilde{D}_1=\roundup{n\pi^*(K_X)|_F-\frac{1}{j}E_{m_1}|_F-H_1}|_C$ with
	$$\deg(\tilde{D}_1)\geq (n-\frac{m_1}{j}-\frac{1}{\beta})\xi+
	\frac{\tilde{\delta}}{j}>2.$$
	Thus $\varphi_{n+1, X}$ is birational, which implies Item (i).
	
	For Item (ii)', even if $n=\frac{m_1}{j}+\frac{1}{\beta}$ is integral,
	the $\bQ$-divisor
	$$n\pi^*(K_X)|_F-\frac{1}{j}E_{m_1}|_F-H_1-C$$
	is nef and big since $S_1|_F$ is nef and big, we still have  $\deg(\tilde{D}_1)>2$.
	Hence $\varphi_{\roundup{\frac{m_1}{j}+\frac{1}{\beta}}+1, X}$ is birational.
	
	A direct application of the above argument implies that, whenever $\deg(\tilde{D}_1)>1$,
	$|K_C+\tilde{D}_1|$ is base point free and so
	$$(n+1)\xi\geq \deg(\tilde{D}_{m_0})+2,$$
	which proves Item (iv).
	
	{}Finally, modulo a further birational modification, we may and do assume that the linear system $$|M_{2j-1}'|=\Mov|(2j-1)(K_{X'}+F)|$$ is base point free. It is clear that $M_{2j-1}'$ is big. By the vanishing theorem and Theorem \ref{KaE}, we have
	\begin{eqnarray*}
		|(2j+2m_1)K_{X'}||_F&\lsgeq& |K_{X'}+M_{2j-1}'+2S_1+F||_F\\
		&\lsgeq&|K_F+(2j-1)\sigma^*(K_{F_0})+2S_1|_F|\\
		&\lsgeq&|2j\sigma^*(K_{F_0})+2S_1|_F|,
	\end{eqnarray*}
	which directly implies the statement in Item (iii).

	Statement $(i)'$ follows from the similar argument to that for (i).  Instead of using the relation
	\eqref{FF1}, one may use the statement (iii), namely:
	$$\frac{j+m_1}{j}\pi^*(K_X)|_F\equiv C+\frac{1}{j}S_1|_F+H''$$
	for an effective $\bQ$-divisor $H''$ on $F$.
	In fact, it suffices to take
	$$n=\rounddown{\frac{2}{\xi(|G|)}(1-\frac{\tilde{\delta}}{j})+\frac{2m_1}{j}}+2$$ 
	to obtain (i).

	We are left to treat Item (ii). By  Kawamata-Viehweg vanishing theorem, we have
	\begin{align}\label{ma1}
	|(n+1)K_{X'}||_F &\lsgeq |K_{X'}+\roundup{n\pi^*(K_X)-\frac{2}{\tilde{\delta}}E_{m_1}-\frac{1}{\mu}E_F^{'}}||_F \notag \\
	&\lsgeq |K_F+\roundup{n\pi^*(K_X)|_F-\frac{2}{\tilde{\delta}}E_{m_1}|_F-\frac{1}{\mu}\cdot (1-\frac{2j}{\tilde{\delta}})E_F^{'}|_F}|
	\end{align}
	since $$n\pi^*(K_X)-\frac{2}{\tilde{\delta}}E_{m_1}-\frac{1}{\mu}\cdot (1-\frac{2j}{\tilde{\delta}})E_F^{'}\equiv (n-\frac{2m_1}{\tilde{\delta}}-\frac{1}{\mu}\cdot(1-\frac{2j}{\tilde{\delta}}))\pi^*(K_X)+\frac{2}{\tilde{\delta}}S_1$$
	is simple normal crossing (by our assumption), nef and big.
	
	Then the vanishing theorem on $F$ gives
	\begin{align}\label{m022}
	&\ \ \ \ |K_F+\roundup{n\pi^*(K_X)|_F-\frac{2}{\tilde{\delta}}E_{m_1}|_F-\frac{1}{\mu}\cdot (1-\frac{2j}{\tilde{\delta}})E_F^{'}|_F}| \notag \\
	&\lsgeq |K_F+\roundup{n\pi^*(K_X)|_F-\frac{2}{\tilde{\delta}}E_{m_1}|_F-\frac{1}{\mu}\cdot (1-\frac{2j}{\tilde{\delta}})E_F^{'}|_F-H_1}| \notag \\
	&=|K_C+\tilde{D_n}|
	\end{align}
	where $\tilde{D_n}=\roundup{n\pi^*(K_X)|_F-\frac{2}{\tilde{\delta}}E_{m_1}|_F-\frac{1}{\mu}\cdot (1-\frac{2j}{\tilde{\delta}})E_F^{'}|_F-H_1}|_C$ with $\mathrm{deg}\tilde{D_n}>2$. Hence $$\varphi_{\rounddown{\frac{2m_1}{\tilde{\delta}}+\frac{1}{\beta}+\frac{1}{\mu}\cdot(1-\frac{2j}{\tilde{\delta}})}+2,X}$$ is birational.
\end{proof}

\begin{thm}\label{k3} Let $X$ be a minimal projective $3$-fold of general type with $p_g(X)\geq 2$. Keep Assumption ($\pounds$).  Let $m_1$ be a positive integer.   Suppose that $M_{m_1}\geq j_1F+S_1$ for some moving divisor $S_1$ on $X'$, $j_1>0$ is an integer and that
	${S_1}|_F\geq j_2C+C'$ where $C'$ is a moving irreducible curve on $F$ with $C'\not\equiv C$, $j_2>0$ is an integer. Set $\delta_2=(C'\cdot C)$.  The following statements hold:
	\begin{itemize}
		\item[(i)]  if
		$j_1\geq j_2$,  then
		\begin{itemize}
			\item[(i.1)]{{ For any positive integer $n$ satisfying $n>\frac{m_1}{j}+\frac{j_1-j_2}{j_1}\cdot\frac{1}{\beta}$ and $$(n-\frac{m_1}{j_1}-\frac{j_1-j_2}{j_1}\cdot \frac{1}{\beta})\xi+\frac{\delta_2}{j_1}>1,$$ one has
					$$(n+1)\xi\ge \ulcorner(n-\frac{m_1}{j_1}-\frac{j_1-j_2}{j_1}\cdot\frac{1}{\beta})\xi+\frac{\delta_2}{j_1}\urcorner+2.$$}}
			\item[(i.2)] Either $$\varphi_{\rounddown{\frac{1}{\xi(|G|)}(2-\frac{\delta_2}{j_1})+\frac{m_1}{j_1}+\frac{j_1-j_2}{j_1}\cdot \frac{1}{\beta}}+2,X}\ (\text{when}\  \delta_2< 2j_1)$$
			or{{
					$$\varphi_{\rounddown{\frac{2m_1}{\delta_2}+\frac{1}{\mu}\cdot(1-\frac{2j_1}{\delta_2})+\frac{1}{\beta}\cdot(1-\frac{2j_2}{\delta_2})}+2,X}\ (\text{when}\  \delta_2\ge2j_1)$$}}
			is birational.
		\end{itemize}
		
		\item[(ii)] if $j_2\geq j_1$, then
		\begin{itemize}
			\item[(ii.1)]{{ For any positive integer $n>\frac{m_1}{j_2}+\frac{1}{\mu}\cdot(1-\frac{j_1}{j_2})$ and
					$$(n-\frac{m_1}{j_2}-\frac{1}{\mu}\cdot(1-\frac{j_1}{j_2}))\cdot\xi+\frac{\delta_2}{j_2}>1,$$ one has
					$$(n+1)\xi\ge\ulcorner (n-\frac{m_1}{j_2}-\frac{1}{\mu}\cdot(1-\frac{j_1}{j_2}))\cdot\xi+\frac{\delta_2}{j_2}\urcorner+2.$$}}
			\item[(ii.2)] Either
			$$\varphi_{\rounddown{\frac{1}{\xi(|G|)}(2-\frac{\delta_2}{j_2})+\frac{m_1}{j_2}+\frac{1}{\mu}\cdot \frac{j_2-j_1}{j_2}}+2,X}\ (\text{when}\  \delta_2< 2j_2)$$
			or
			{{
					$$\varphi_{\rounddown{\frac{1}{\mu}\cdot(1-\frac{2j_1}{\delta_2})+\frac{1}{\beta}\cdot(1-\frac{2j_2}{\delta_2})}+2,X}\ (\text{when}\   \delta_2\ge 2j_2)$$
					is birational.}}
		\end{itemize}
	\end{itemize}
\end{thm}
\begin{proof} Modulo further birational modification, we may and do assume that $|S_1|$ is also base point free.  Hence $S_1$ is nef.  By assumption we may find two effective $\bQ$-divisor $\tilde{E}'_{m_1}$ and $\tilde{E}''_{m_1}$ such that
	$$m_1\pi^*(K_X)\equiv j_1F+S_1+\tilde{E}'_{m_1},$$
	$$S_1|_F\equiv j_2C+C'+\tilde{E}''_{m_1}.$$
	
	Set
	$$n=\begin{cases}
	\rounddown{\frac{1}{\xi(|G|)}(2-\frac{\delta_2}{j_1})+\frac{m_1}{j_1}+\frac{j_1-j_2}{j_1}\cdot \frac{1}{\beta}}+1, & \text{when}\
	\delta_2\leq 2j_1;\\
	\rounddown{\frac{2m_1}{\delta_2}+\frac{1}{\mu}\cdot(1-\frac{2j_1}{\delta_2})+\frac{1}{\beta}\cdot(1-\frac{2j_2}{\delta_2})}+1,  & \text{when}\  \delta_2>2j_1.
	\end{cases}$$
	By the vanishing theorem, one has
	\begin{eqnarray*}
		|(n+1)K_{X'}||_F& \lsgeq & |K_{X'}+\roundup{n\pi^*(K_X)-\frac{1}{j_1} \tilde{E}'_{m_1}}||_F\\
		&\lsgeq & |K_F+\roundup{n\pi^*(K_X)|_F-\frac{1}{j_1}{\tilde{E}'_{m_1}}|_F}|\\
		&\lsgeq &  |K_F+ \roundup{Q_{1,m_1}}|,\end{eqnarray*}
	where
	\begin{eqnarray*}
		Q_{1,m_1}&=&n\pi^*(K_X)|_F-\frac{1}{j_1}{\tilde{E}'_{m_1}}|_F-
		\frac{1}{j_1}\tilde{E}''_{m_1}\\
		&\equiv& (n-\frac{m_1}{j_1})\pi^*(K_X)|_F+\frac{j_2}{j_1}C+\frac{1}{j_1}C'.
	\end{eqnarray*}
	Recall that we have
	$$\frac{1}{\beta}\pi^*(K_X)|_F\equiv C+H_1,$$
	where $H_1$ is an effective $\bQ$-divisor.
	Hence
	\begin{eqnarray*}
		&&Q_{1,m_1}-\frac{j_1-j_2}{j_1}H_1\\
		&\equiv & \Big(n-\frac{m_1}{j_1}-\frac{j_1-j_2}{j_1}\cdot \frac{1}{\beta}\Big)\pi^*(K_X)|_F+\frac{1}{j_1}C'+C.
	\end{eqnarray*}
	By the vanishing theorem once more, we have
	$$|K_F+\roundup{Q_{1,m_1}-\frac{j_1-j_2}{j_1}H_1}||_C =
	|K_C+D_{1,m_1}|$$
	where
	$$D_{1,m_1}=\roundup{Q_{1,m_1}-\frac{j_1-j_2}{j_1}H_1-C}|_C$$ with
	$$\deg(D_{1,m_1})\geq \Big(n-\frac{m_1}{j_1}-\frac{j_1-j_2}{j_1}\cdot \frac{1}{\beta}\Big)\xi+\frac{\delta_2}{j_1}.$$
	Now, for the similar reason to previous ones, we see that $\varphi_{n+1, X}$ is brational and Statement (i.2) ($\delta_2\le 2j_1$) follows. Statement (i.1) follows similarly.
	
	Now turn to (i.2) where $\delta_2>2j_1$. By Kawamata-Viehweg vanishing theorem, we have
	\begin{align*}
	|(n+1)K_{X'}||_F &\lsgeq |K_{X'}+\roundup{n\pi^*(K_X)-\frac{2}{\delta_2}\tilde{E}_{m_1}^{'}-\frac{1}{\mu}(1-\frac{2j_1}{\delta_2})E_F^{'}}||_F \\
	&\lsgeq |K_F+\roundup{n\pi^*(K_X)|_F-\frac{2}{\delta_2}\tilde{E}_{m_1}^{'}|_F-\frac{1}{\mu}(1-\frac{2j_1}{\delta_2})E_F^{'}|_F}|
	\end{align*}
	since $n\pi^*(K_X)-\frac{2}{\delta_2}\tilde{E}_{m_1}^{'}-\frac{1}{\mu}(1-\frac{2j_1}{\delta_2})E_F^{'}$ is simple normal crossing (by definition), nef and big. By vanishing theorem on $F$, we have
	\begin{align*}
	&|K_F+\roundup{n\pi^*(K_X)|_F-\frac{2}{\delta_2}\tilde{E}_{m_1}^{'}|_F-\frac{1}{\mu}(1-\frac{2j_1}{\delta_2})E_F^{'}|_F}||_C\\
	&\lsgeq |K_F+\roundup{ n\pi^*(K_X)|_F-\frac{2}{\delta_2}\tilde{E}_{m_1}^{'}|_F-\frac{1}{\mu}(1-\frac{2j_1}{\delta_2})E_F^{'}|_F-\frac{2}{\delta_2}\tilde{E}_{m_1}^{''}- (1-\frac{2j_2}{\delta_2})H_1}|_C \\
	& \lsgeq |K_C+\roundup{\tilde{D}_{1,m_1}}|
	\end{align*}
	where $\tilde{D}_{1, m_1}\equiv(n-\frac{2m_1}{\delta_2}-\frac{1}{\mu}(1-\frac{2j_1}{\delta_2})-\frac{1}{\beta}(1-\frac{2j_2}{\delta_2}))\pi^*(K_X)|_C+\frac{2}{\delta_2}C'|_C$ with $\mathrm{deg}(\tilde{D}_{1,m_1})>2$. Thus $\varphi_{n+1, X}$ is birational.
	
	One gets Statement (ii) in the similar way, which we leave to interested readers.
\end{proof}

\begin{rem}
	{}From the proof of Theorem \ref{k3}, one clearly sees that the variant of Theorem \ref{k3} with $C'=0$ is also true.
\end{rem}

\begin{thm}\label{kf} Let $X$ be a minimal projective $3$-fold of general type with $p_g(X)\geq 2$.  Keep Assumption ($\pounds$).  Let $m_1$ be a positive integer.   Suppose that $M_{m_1}\geq j_1F+S_1$ for some moving divisor $S_1$ on $X'$, $j_1>0$ is an integer and that
	${S_1}|_F\geq j_2C$ where $j_2>0$ is an integer. The following statements hold:
	\begin{itemize}
		\item[(i)]  if
		$j_1\geq j_2$,  then
		\begin{itemize}
			\item[(i.1)] For any positive integer $n$ satisfying $n>\frac{m_1}{j}+\frac{j_1-j_2}{j_1}\cdot \frac{1}{\beta}$ and $$(n-\frac{m_1}{j_1}-\frac{j_1-j_2}{j_1}\cdot \frac{1}{\beta})\xi>1,$$ one has
			$$(n+1)\xi\ge \ulcorner(n-\frac{m_1}{j_1}-\frac{j_1-j_2}{j_1}\cdot\frac{1}{\beta})\xi\urcorner+2.$$
			\item[(i.2)] The map $$\varphi_{\rounddown{\frac{2}{\xi(|G|)}+\frac{m_1}{j_1}+\frac{j_1-j_2}{j_1}\cdot \frac{1}{\beta}}+2,X}$$
			is birational.
		\end{itemize}
		\item[(ii)] if $j_2\geq j_1$, then
		\begin{itemize}
			\item[(ii.1)] For any positive integer $n>\frac{m_1}{j_2}+\frac{1}{\mu}\cdot(1-\frac{j_1}{j_2})$ and
			$$(n-\frac{m_1}{j_2}-\frac{1}{\mu}\cdot(1-\frac{j_1}{j_2}))\cdot\xi>1,$$ one has
			$$(n+1)\xi\ge\ulcorner (n-\frac{m_1}{j_2}-\frac{1}{\mu}\cdot(1-\frac{j_1}{j_2}))\cdot\xi\urcorner+2.$$
			\item[(ii.2)] The map
			$$\varphi_{\rounddown{\frac{2}{\xi(|G|)}+\frac{m_1}{j_2}+\frac{1}{\mu}\cdot \frac{j_2-j_1}{j_2}}+2,X}$$
			is birational.
		\end{itemize}
	\end{itemize}
\end{thm}
\begin{proof} This follows directly from the proof of Theorem \ref{k3} by blinding $C'$. We omit the redundant details.
\end{proof}

%%%%%%%%%%%%%%%%%%%%
%%%%%%%%% PG=2 %%%%%%%%%%%%%%%%%
\section{\bf Threefolds of general type with $p_g=2$}

This section is devoted to the classification of 3-folds of general type with $p_g(X)=2$. Keep the same notation as in \ref{setup}. We have an induced fibration $f\colon X'\lrw \Gamma$ of which the general fiber $F$ is a nonsingular projective surface of general type. Denote by $\sigma\colon F\rightarrow F_0$ the contraction onto its minimal model.

By \cite[Theorem 3.3]{IJM}, it is sufficient to assume that $b=g(\Gamma)=0$, i.e. $\Gamma\cong \bP^1$. Since $p_g(X)>0$ and $F$ is a general fiber, we have $p_g(F)>0$. By the surface theory, $F$ belongs to one of the $3$ types:
\begin{enumerate}
\item $(K_{F_0}^2, p_g(F_0))=(1,2)$;
\item $(K_{F_0}^2, p_g(F_0))=(2,3)$;
\item other surfaces with $p_g(F_0)>0$.
\end{enumerate}

 Recall that we have
\begin{align}\label{eq:zeta}
 \pi^*(K_X)\sim_{\bQ}F+E_{1}^{'},
\end{align}
where  $E_{1}^{'}$ is an effective $\mathbb{Q}$-divisor since $\mu\geq 1$.
By Theorem \ref{KaE}, the natural restriction map
\begin{align}\label{eq:adjunction}
H^0(X', 3K_{X'}+3F)\rightarrow H^0(F, 3K_F)
\end{align}
is surjective. Since $|6K_{X'}|\lsgeq |3K_{X'}+3F|$,
by \eqref{eq:zeta} and \eqref{eq:adjunction}, we may write
\begin{align}\label{eq:restriction}
\pi^{*}(K_X)|_F\equiv \frac{1}{2}\sigma^*(K_{F_0})+Q^{'}\equiv \frac{1}{2}C+\hat{E}_F,
\end{align}
where $Q^{'}$ and $\hat{E}_F$ are effective $\bQ$-divisors.

\subsection{Non-$(1,2)$-surface case}

\begin{lem}\label{lem:phi6 general case}
Let $X$ be a minimal projective $3$-fold of general type with $p_g(X)=2$ and keep the setting in \ref{setup}. Suppose $d_1=1$, $\Gamma\cong \bP^1$ and $F$ is neither a $(2,3)$ surface nor a $(1,2)$-surface. Then $\varphi_{6,X}$ is birational.
\end{lem}
\begin{proof} As $|6K_{X'}|$ distinguishes different general fibers of $f$, \eqref{eq:adjunction}
 implies that $\varphi_{6,X}$ is birational unless $F$ is either a  $(1,2)$-surface or a $(2,3)$-surface.
\end{proof}

\begin{prop}\label{prop:(2,3) surface}
Let $X$ be a minimal projective $3$-fold of general type with $p_g(X)=2$, $d_1=1$, $\Gamma\cong \bP^1$. Assume that $F$ is a $(2,3)$-surface. Then $\varphi_{6,X}$ is not birational if and only if $(\pi^*(K_X)|_F)^2=\frac{1}{2}$.
\end{prop}

\begin{proof}
We have $(\pi^*(K_X)|_F)^2\ge\frac{1}{2}$ by (\ref{eq:restriction}) and by our assumption. By (\ref{eq:zeta}) and Kawamata-Viehweg vanishing theorem, we have
\begin{align}\label{eq:(2,3)adjunction}
|6K_{X'}||_F&\lsgeq |K_{X'}+F+\ulcorner4\pi^*(K_X)\urcorner||_F\nonumber \\
&\lsgeq |K_F+\ulcorner4\pi^*(K_X)|_F\urcorner|.
\end{align}
If $(\pi^*(K_X)|_F)^2>\frac{1}{2}$, then $|K_F+\ulcorner4\pi^*(K_X)|_F\urcorner|$ is birational by \cite[Lemma 2.3]{CHP}, \cite[Lemma 2.5]{CHP} and \eqref{eq:restriction}. So $\varphi_{6,X}$ is birational.

If $(\pi^*(K_X)|_F)^2=\frac{1}{2}$,  we have $\pi^*(K_X)|_F\equiv \frac{1}{2}\sigma^*(K_{F_0})$ by Hodge index theorem and \eqref{eq:restriction}.  It is clear that  $|3\sigma^*(K_{F_0})||_C\lsleq |K_C+C|_C|$.
Note that $|C|=|\sigma^*(K_{F_0})|$ is base point free,  we have
\begin{align*}
{|6K_{X'}|}|_C&\lsgeq {|\llcorner6\pi^*(K_X)\lrcorner||_F}|_C\\
&\lsgeq |3\sigma^*(K_{F_0})||_C. \end{align*}
Since $(\pi^*(K_X)\cdot C)=1$, the vanishing theorem and \eqref{eq:(2,3)adjunction} implies that ${|M_6|}|_C=|K_C+D|$ with $\deg(D)\geq 2$.  Hence ${|M_6|}|_C=|K_C+C|_C|$.
Since $|C|_C|$ gives a $g^{1}_{2}$, $|K_C+C|_C|$ is clearly non-birational.

Therefore $\varphi_{6,X}$ is non-birational if $(\pi^*(K_X)|_F)^2=\frac{1}{2}$.
\end{proof}

\subsection{The $(1,2)$-surface case} \

{}From now on, we always assume that $F$ is a $(1,2)$-surface. We have $0\le \chi(\omega_X)\le 1$ by our assumption and \cite[Lemma 4.5]{MCNoe}. It is well known that $|K_{F_0}|$ has exactly one base point and that, after blowing up this point, $F$ admits a canonical fibration of genus $2$ with a unique section which we denote by $H$. Denote by $C$ a general member in $|G|=\Mov|{\sigma^*(K_{F_0})}|$.

\begin{rem}\label{rem:xi=2/3}
 By \cite[Theorem 1.1]{C14}, $\varphi_{7,X}$ is non-birational if and only if $\xi=\frac{2}{3}$. Since $p_g(X)>0$, $\varphi_{6,X}$ is non-birational as well if $\xi=\frac{2}{3}$.
\end{rem}

\begin{lem}\label{lem: distin}
	Let $X$ be a minimal projective $3$-fold of general type with $p_g(X)=2$, $d_1=1$, $\Gamma\cong \bP^1$. Assume that $F$ is a $(1,2)$-surface. Keep the setting in \ref{setup}. Then $|6K_{X'}|$ distinguishes different generic irreducible elements of $|M|$ and $|6K_{X'}||_F$ distinguishes generic irreducible elements of $|G|$. 
\end{lem}

\begin{proof}
	Since $|M|$ is composed of a rational pencil, it is clear that $|6K_{X'}|$ distinguishes different generic irreducible elements of $|M|$. Notice that $|G|$ is composed of a rational pencil. The surjectivity of \eqref{eq:adjunction} implies that $|6K_{X'}||_F$ distinguishes generic irreducible elements of $|G|$.
\end{proof}

\begin{lem}\label{lem:beta} (see \cite[2.15]{MA})
Let $X$ be a minimal projective $3$-fold of general type with $p_g(X)=2$, $d_1=1$, $\Gamma\cong\bP^1$. Assume that  $F$ is a $(1,2)$-surface. Then $ \beta\ge\frac{1}{2}$ and $ \xi\ge\frac{2}{3}.$
\end{lem}

\begin{lem}\label{lem:restric vol}
Under the same condition as that of Lemma \ref{lem:beta}. Assume that $\xi=1$ and $(\pi^*(K_X)|_F)^2>\frac{1}{2}$. Then $\varphi_{6,X}$ is birational.
\end{lem}
\begin{proof}
Consider the Zariski decomposition of the following $\mathbb{Q}$-divisor:
$$2\pi^*(K_X)|_F+2\hat{E}_F\equiv (2\pi^*(K_X)|_F+N^{+})+N^{-},$$
where
\begin{enumerate}
\item both $N^{+}$ and $N^{-}$ are effective $\mathbb{Q}$-divisors with $N^{+}+N^{-}=2\hat{E}_F$;
\item the $\mathbb{Q}$-divisor $(2\pi^*(K_X)|_F+N^{+})$ is nef;
\item $((2\pi^*(K_X)|_F+N^{+})\cdot N^{-})=0$.
\end{enumerate}
\noindent{\bf Step 1}.  $(\pi^*(K_X)|_F)^2>\frac{1}{2}$ implies $(N^+ \cdot C) > 0$.

Since $C$ is nef, we see $(N^+ \cdot C) \geq 0$. Assume the contrary that $(N^+ \cdot C) =0$.  Then $(N^+)^2\leq 0$ as $C$ is a fiber of the canonical fibration of $F$.
Notice that
$$\frac{1}{2}<(\pi^*(K_X)|_F)^2=\frac{1}{2}(\pi^*(K_X)\cdot C)+(\pi^*(K_X)|_F\cdot \hat{E}_F)$$
implies $(\pi^*(K_X)|_F\cdot \hat{E}_F)>0$. We clearly have $(\pi^*(K_X)|_F\cdot N^+)>0$ by the definition of Zariski decomposition.  Hence
\begin{eqnarray*}
(N^+)^2&=&\big(N^+\cdot (2\pi^*(K_X)|_F-C-N^-)\big)\\
&=&2(N^+\cdot \pi^*(K_X)|_F)+(2\pi^*(K_X)|_F\cdot N^-)>0,
\end{eqnarray*}
a contradiction.
\medskip

\noindent{\bf Step 2}. $(N^+\cdot C)>0$ implies the birationality of $\varphi_{6,X}$.

By Kawamata-Viehweg vanishing theorem, we have
\begin{eqnarray*}
|6K_{X'}||_F&\lsgeq&|K_{X'}+\roundup{5\pi^*(K_X)-E_1'}||_F\\
&\lsgeq& |K_F+\roundup{(5\pi^*(K_X)-E_1')|_F}.
\end{eqnarray*}
Noting that
\begin{eqnarray}
(5\pi^*(K_X)-E_1')|_F&\equiv& 4\pi^*(K_X)|_F\equiv 2\pi^*(K_X)|_F+C+2\hat{E}_F\notag\\
&\equiv& (2\pi^*(K_X)|_F+N^+)+C+N^-,\label{eq:restric vol 1}
\end{eqnarray}
and that $2\pi^*(K_X)|_F+N^+$ is nef and big, the vanishing theorem gives
\begin{equation}
|K_F+\roundup{(5\pi^*(K_X)-E_1')|_F-N^-}||_C=|K_C+D^+|,\label{eq:restrict vol 2}
\end{equation}
where $\deg(D^+)\geq 2\xi+(N^+\cdot C)>2$.  By Lemma \ref{lem: distin}, \eqref{eq:restric vol 1} and \eqref{eq:restrict vol 2}, $\varphi_{6,X}$ is birational.
\end{proof}

\begin{lem}\label{lem:beta2}
Under the same condition as that of Lemma \ref{lem:beta}. If $\beta>\frac{2}{3}$, then $\varphi_{6,X}$ is birational.
\end{lem}
\begin{proof}
Since $$\alpha(5)\ge(5-1-1-\frac{1}{\beta})\cdot \xi >1,$$ we have $\xi\ge\frac{4}{5}$ by \eqref{kieq1}. Now, as $\alpha(6)>2$, $\varphi_{6,X}$ is birational by Theorem \ref{key-birat}.
\end{proof}

\begin{lem}\label{lem:zeta}
Under the same condition as that of Lemma \ref{lem:beta}. If $\mu>\frac{4}{3}$, then $\varphi_{6,X}$ is birational.
\end{lem}
\begin{proof} By assumption and \eqref{cri}, one has
 $\beta\ge\frac{4}{7}$. So $$\alpha(5)=(5-1-\frac{1}{\mu}-\frac{1}{\beta})\cdot\xi>1,$$ which  implies $\xi\ge\frac{4}{5}$. Since  $\alpha(6)>\frac{5}{2}\cdot\xi\ge 2$,  $\varphi_{6,X}$ is birational for the similar reason.
\end{proof}

\begin{lem}\label{lem:Pm1 bir}
Let $X$ be a minimal projective $3$-fold of general type with $p_g(X)=2$, $d_1=1$ and $\Gamma\cong\bP^1$. Keep the notation in \ref{setup}.  Assume that $F$ is a $(1,2)$-surface.
Let $m_1\ge 2$ be any integer. Then $\varphi_{6,X}$ is birational provided that one of the following holds:
\begin{enumerate}
\item[(i)] $u_{m_1,0}=h^0(F, m_1K_F)$;
\item[(ii)] $h^0(M_{m_1}-jF)\geq \rounddown{\frac{4}{3}m_1}-j+2\ge 2$ and $u_{m_1, -j}\le 1$ for some integer $j\ge 0$.
\end{enumerate}
\end{lem}
\begin{proof}
(i). Since $\theta_{m_1,0}$ is surjective, we have $m_1\pi^*(K_X)|_F\ge m_1 C$, which means $\beta=1$. Hence $\varphi_{6,X}$ is birational by Lemma \ref{lem:beta2}.

(ii). By assumption, $|M_{m_1}-jF|$ and $|F|$ are composed of same pencil. Hence we have $\mu >\frac{4}{3}$ and  $\varphi_{6,X}$ is birational by Lemma \ref{lem:zeta}.
\end{proof}

\begin{lem}\label{lem:charac K1} Under the same condition as that of Lemma \ref{lem:beta}. Suppose that $\xi=\frac{4}{5}$. Then $\varphi_{6,X}$ is not birational.
\end{lem}
\begin{proof}
Write $|M_6|=\mathrm{Mov}|6K_{X'}|$. One may assume that $|M_6|$ is base point free. Denote by $G_{6,0}$ a general member of $|M_6|_F|$. Then $$|G_{6,0}|_C|\lsgeq |G_{6,0}||_C\lsgeq |M_6||_C.$$ Since $\alpha(6)>1$, we have $(6\pi^*(K_X)|_F\cdot C)\ge (G_{6,0}\cdot C)\ge 4 $ by \eqref{kieq1}. Therefore $(G_{6,0}\cdot C)=4$ by the assumption. On the other hand, the vanishing theorem gives
\begin{eqnarray*}
|6K_{X'}||_C &\lsgeq & |K_F+\roundup{4\pi^*(K_X)|_F}||_C\\
&\lsgeq & |K_F+\roundup{4\pi^*(K_X)|_F-2\hat{E}_F}||_C\\
&\lsgeq& |K_F+\hat{D}|,
\end{eqnarray*}
where $\deg(\hat{D})\geq 2$, which forces $|S_{6,0}||_F|_C=|G_{6,0}|_C|$.  Take a general effective divisor $K\in |K_X|$. Then $\mathrm{supp}(\pi^*(K)|_F)|_C$ consists of just one point $P\in C$ since $\deg(\sigma^*(K_{F_0})|_C)=1$. Here the divisor $P$ satisfies $2P\in |K_C|$. So $|4P|=|\llcorner (6\pi^*K|_F)|_C\lrcorner| \lsgeq |G_{6,0}|_C|$. Therefore the restriction of the linear system $|S_{6,0}|$ on $C$ is just $|2K_C|$, which implies that $\varphi_{6,X}$ is not birational.
\end{proof}

\subsection{Effective constraints on $P_2$, $P_3$, $P_4$, $P_5$ and $P_6$}

This subsection is devoted to link some numerical constrains on plurigenera $P_i(X)$ ($i=1, \ldots, 6$) to the birationality of $\varphi_6$. 

The following proposition is the prototype for Propositions \ref{prop:P3}, \ref{prop:P4}, \ref{prop:P5} and \ref{prop:P6}.
\begin{prop}\label{prop:P2}
Under the same condition as that of Lemma \ref{lem:beta}, then
\begin{enumerate}
\item when $P_2(X)\ge 6$, $\varphi_{6,X}$ is birational;
\item when $P_2(X)=5$,  $\varphi_{6,X}$ is not birational if and only if
$$\xi=1\ \ \ \text{and}\ \ \ (\pi^*(K_X)|_F)^2=\frac{1}{2}.$$
 \end{enumerate}
\end{prop}
\begin{proof}
Set $m_1=2$. By Lemma \ref{lem:Pm1 bir}, we may assume $$u_{2,0}\le h^0(2K_F)-1=3.$$

{\bf Case 1.}  $u_{2,0}\leq 3$ and $u_{2,-1}=3$.

There is a moving divisor $S_{2,-1}$ on $X'$ such that
$$M_2\ge F+S_{2,-1}$$
and $h^0(F, S_{2,-1}|_F)\ge3$. Modulo further birational modification, we may assume that $|S_{2,-1}|$ is base point free. Denote by $C_{2,-1}$ a generic irreducible element of $|S_{2,-1}|_F|$. Then $|C_{2,-1}|$ is moving since $q(F)=0$.

If $|S_{2,-1}|_F|$ and $|C|$ are composed of same pencil, then
$$M_2|_F\ge S_{2,-1}|_F\ge 2C,$$
which means $\beta\ge 1$ and $\varphi_{6,X}$ is birational by Lemma \ref{lem:beta2}.

If $|S_{2,-1}|_F|$ and $|C|$ are not composed of the same pencil, then $\varphi_{6,X}$ is birational by \cite[Proposition 3.6]{CHP} (1).
\medskip

{\bf Case 2.} $u_{2,0}\leq 3$ and $u_{2,-1}=2$.

If $|S_{2,-1}|_F|$ and $|C|$ are not composed of the same pencil, then $\varphi_{6,X}$ is birational by \cite[Proposition 3.6]{CHP}(1).

If $|S_{2,-1}|_F|$ and $|C|$ are composed of the same pencil, then we have $\xi\ge\frac{4}{5}$ by \cite[Proposition 3.6]{CHP} (2.1) (as $2\xi( |G|)\ge\frac{4}{3}>1$). By \cite[Proposition 3.6]{CHP} (2.2) (taking $n=3$),  $\varphi_{6,X}$ is birational.
\medskip

{\bf Case 3.}  $u_{2,0}\leq 3$, $u_{2,-1}\le 1$ and $P_2(X)\ge 6$.

We have $h^0(M_2-F)\ge3$. Hence $\varphi_{6,X}$ is birational by Lemma \ref{lem:Pm1 bir} (ii).

Suppose $P_2(X)=5$.

We first assume that $\varphi_{6,X}$ is not birational. By the arguments in {\bf Case 1}-{\bf Case 3}, we have $u_{2,0}\le 3$ and $u_{2,-1}=1$. If $u_{2,0}\le 2$, one has $h^0(M_2-F)\ge 3$ by our assumption. Thus we have $\mu\ge \frac{3}{2}$. Lemma \ref{lem:zeta} implies that $\varphi_{6,X}$ is birational, which is a contradiction. So we have $u_{2,0}=3$. If $|M_2|_F|$ and $|C|$ are composed of the same pencil, we get $\beta\ge 1$. Lemma \ref{lem:beta2} implies that $\varphi_{6,X}$ is birational. So $|M_2|_F|$ and $|C|$ are not composed of the same pencil. Thus we have $$\xi=(\pi^*(K_X)|_F\cdot C)\ge\frac{1}{2}( M_2|_F\cdot C)\ge 1.$$ Lemma \ref{lem:restric vol} implies that $(\pi^*(K_X)|_F)^2=\frac{1}{2}$.

Conversely, assume that $\xi=1$ and $(\pi^*(K_X)|_F)^2=\frac{1}{2}$. Observing that $\beta>\frac{1}{2}$ induces $(\pi^*(K_X)|_F)^2>\frac{1}{2}$. Since we assume that $\xi=1$, by the argument in {\bf Case 1}-{\bf Case 3}, one has $u_{2,0}\le 3$ and $u_{2,-1}=1$. If $u_{2,0}\le 2$, we have $h^0(F, M_2-F)\ge 3$. By  \cite[Proposition 3.5]{CHP}, one has $\beta>\frac{3}{5}$, which contradicts to our assumption. So we have $u_{2,0}=3$. Similarly $|M_2|_F|$ induces a generically finite morphism. Pick a generic irreducible element $C_2$ in $|M_2|_F|$.

Since $$(\pi^*(K_X)|_F)^2\ge\frac{1}{2}(\pi^*(K_X)|_F\cdot C_2)\ge\frac{1}{4}C_2^2,$$ we have $C_2^2=2$ and $(\pi^*(K_X)|_F\cdot C_2)=1$.  Since $|C_2|_{C_2}|$ is generically finite,  $|C_2|_{C_2}|$ is a $g^1_{2}$.  By Kawamata-Viehweg vanishing theorem, we have
\begin{align*}
|6K_{X'}||_F&\lsgeq|K_{X'}+F+2M_2||_F\lsgeq |K_F+2C_2|
\end{align*}
and
\begin{align*}
|K_F+2C_2||_{C_2}=|K_{C_2}+C_2|.
\end{align*}
So $|M_6||_{C_2}\lsgeq |K_{C_2}+C_{2}|_{C_2}|$. Since $6=(6\pi^*(K_X)|_F\cdot C_{2})\ge \mathrm{deg}(M_6|_{C_2})$, we have $$|M_6|_{C_2}|=|M_6||_{C_2}=|K_{C_2}+C_{2}|_{C_2}|.$$ $|C_2|_{C_2}|$ is $g^1_2$ implies that $\varphi_{6,X}$ is not birational.
\end{proof}

The interested  reader should  read carefully the proof of the next following four Propositions . Otherwise, one can possibly skip the proofs knowing only that they on the same lines as the previous one, with an increasing order of cases and difficulties.

\begin{prop}\label{prop:P3}
Under the same condition as that of Lemma \ref{lem:beta}, if $P_3(X)\ge 9$, then $\varphi_{6,X}$ is birational.
\end{prop}
\begin{proof}
Set $m_1=3$. By Lemma \ref{lem:Pm1 bir}, we may assume that $u_{3,0}\le h^0(3K_F)-1=5$. Since $|3K_{X'}|\lsgeq |K_{X'}+F_1+F|$ for two distinct general fibers of $f$, it follows from \cite[Lemma 4.6]{MCNoe} that $$|M_3||_F\lsgeq |C|,$$ which means $u_{3,0}\ge 2$.
\medskip

{\bf Case 1.} $u_{3,0}=5$.

We have $h^0(M_3|_F-C)\ge 3$ since $v_{3,0}\leq 2$ by \cite[Proposition 3.4]{CHP}.

If $v_{3,-1}\ge 2$, we have $M_3|_F\ge C+C_{3,-1}$ where $C_{3,-1}$ is a moving curve on $F$ with $(C_{3,-1}\cdot C)\geq 2$. Hence $\varphi_{6,X}$ is birational by \cite[Proposition 3.7]{CHP}(i).

If $v_{3,-1}\le 1$, then $|M_3|_F-C|$ and $|C|$ are composed of the same pencil. Then $M_{3}|_F\ge 3C$ which means $\beta\ge 1$ and so $\varphi_{6,X}$ is birational by Lemma \ref{lem:beta2}.
\medskip

{\bf Case 2.}  $u_{3,0}=4$

We have $h^0(M_3|_F-C)\ge 2$.

If $v_{3,-1}\ge 2$, we get $M_3|_F\ge C+C_{3,-1}$ where $C_{3,-1}$ is a moving curve on $F$ with $(C_{3,-1}\cdot C)\geq 2$. Hence $\varphi_{6,X}$ is birational by \cite[Proposition 3.7]{CHP}(i).

If $v_{3,-1}\leq 1$, we see that $|M_3|_F-C|$ and $|C|$ are composed of the same pencil. Then $M_3|_F\ge 2C$ which implies $\beta\ge \frac{2}{3}$.
\medskip

\noindent{Subcase 2.1}. $u_{3,-1}\ge 3$.

If $|S_{3,-1}|_F|$ and $|C|$ are not composed of the same pencil,  then $(S_{3,-1}|_F\cdot C)\geq 2$ and so $\varphi_{6,X}$ is birational by \cite[Proposition 3.6]{CHP}(1.2) ($m_1=3$, $j=1$, $\beta\ge\frac{2}{3}$, $\mu=1$, $\tilde{\delta}=2$).

If $|S_{3,-1}|_F|$ and $|C|$ are  composed of the same pencil, then $S_{3,-1}|_F\ge 2C$. We get $\beta\ge\frac{3}{4}$ by  \cite[Proposition 3.5]{CHP} ($n_1=3$, $j_1=1$, $l_1=2$) and so $\varphi_{6,X}$ is birational by Lemma \ref{lem:beta2}.
\medskip

\noindent{Subcase 2.2}. $u_{3,-1}\le 2$ and $u_{3,-2}=2$.

If $|S_{3,-2}|_F|$ and $|C|$ are not composed of the same pencil, $\varphi_{6,X}$ is birational by Theorem \ref{k2}(i).

If $|S_{3,-2}|_F|$ and $|C|$ are composed of the same pencil, we hope to use Theorem \ref{kf}
with $j_1=2$ and $j_2=1$. Recall that $\beta\geq \frac{2}{3}$, $\xi\geq \frac{2}{3}$ and $\mu=1$. By taking $n=7$ and applying Inequality \eqref{kieq1}, we get $\xi\geq \frac{5}{7}$. Similarly, one gets $\xi\geq \frac{4}{5}$ by one more step optimization. {}Finally Theorem \ref{kf}(i) implies the birationality of $\varphi_{6,X}$.
\medskip

\noindent{Subcase 2.3}. $u_{3,-1}\le 2$, $u_{3,-2}=1$ and $P_3(X)\ge 9$.

One has $h^0(M_3-2F)\geq 3$.  As $u_{3,-2}=1$, one has  $M_3\ge 4F$, which implies
$\mu\geq \frac{4}{3}$. As $\alpha(5)>1$, we get $\xi\geq \frac{4}{5}$ by \eqref{kieq1}.
Since $\alpha(6)>2$, $\varphi_{6,X}$ is birational.
\medskip

{\bf Case 3.} $u_{3,0}\le 3$, $u_{3,-1}\le 3$ and $u_{3,-2}\ge 2$.

We have $M_3\ge 2F+S_{3,-2}$ for a moving divisor $S_{3,-2}$ with $h^0(S_{3,-2}|_F)\geq 2$. Clearly $\beta\geq \frac{2}{3}$.

If $|S_{3,-2}|_F|$ and $|C|$ are not composed of the same pencil, then $\xi\ge\frac{4}{5}$ by Theorem \ref{k2}(iv) ($n=4$). Hence $\varphi_{6,X}$ is birational by Theorem \ref{k2}(i).

If $|S_{3,-2}|_F|$ and $|C|$ are composed of the same pencil, we get $\xi\ge\frac{4}{5}$ by Theorem \ref{kf}(i.1) ($n=4$)  and $\varphi_{6,X}$ is birational by Theorem \ref{kf}(i.2).
\medskip

{\bf Case 4.} $u_{3,0}\le 3$, $u_{3,-1}=3$, $u_{3,-2}=1$ and $P_3(X)\ge 9$.

We have $h^0(M_3-2F)\geq 3$.  Since $u_{3,-2}=1$, we have $M_3\ge 4F$. Thus, by  \cite[Proposition 3.5]{CHP}, we have $\beta\geq \frac{4}{7}$.

If $|S_{3,-1}|_F|$ and $|C|$ are not composed of the same pencil, we have $(S_{3,-1}|_F\cdot C)\ge 2$. \cite[Proposition 3.6]{CHP}(1.2) implies the birationality of $\varphi_{6,X}$ ($m_1=3$, $j=1$, $\delta=2$, $\mu=\frac{4}{3}$).

If $|S_{3,-1}|_F|$ and $|C|$ are composed of the same pencil,  we have $\beta\geq \frac{3}{4}$ by  \cite[Proposition 3.5]{CHP}. Thus $\varphi_{6,X}$ is birational by Lemma \ref{lem:beta2}.
\medskip

{\bf Case 5.} $u_{3,0}\le 3$, $u_{3,-1}\le 2$, $u_{3,-2}=1$ and $P_3(X)\ge 9$.

Clearly we have $h^0(M_3-2F)\ge 4$, which implies that $\mu\geq \frac{5}{3}$. Hence $\varphi_{6,X}$ is birational by Lemma \ref{lem:zeta}.
\end{proof}

\begin{prop}\label{prop:P4}
Under the same condition as that of Lemma \ref{lem:beta}, then
\begin{enumerate}
\item when $P_4(X)\ge 15$,  $\varphi_{6,X}$ is birational;
\item when $P_4(X)=14$, $\varphi_{6,X}$ is non-birational if and only if one of the following holds:
\begin{enumerate}
\item[(2.1)] $\xi=\frac{4}{5}$;
\item[(2.2)] $\xi=1$ and $(\pi^*(K_X)|_F)^2=\frac{1}{2}$.
\end{enumerate}
\end{enumerate}
\end{prop}

\begin{proof}
Set $m_1=4$. By Lemma \ref{lem:beta2}, we may and do assume that $u_{4,0}\leq h^0(4K_F)-1=8$.

By \cite[Proposition 3.4]{CHP}(1), we know that $v_{4,0}\leq 3$. We claim that  $v_{4,-1}=3$ is impossible. Otherwise, we have $M_4|_F\geq C+C_{-1}$, where $C_{-1}$ is a moving curve on $F$ satisfying $h^0(C, C_{-1}|_C)\geq 3$. In particular, one has $(C_{-1}\cdot C)\geq 4$ by Riemann-Roch formula. Now we have
$$4\geq (\sigma^*(K_{F_0})\cdot M_4|_F)\geq (\sigma^*(K_{F_0})\cdot (C+C_{-1}))\geq 5,$$
which is a contradiction.  Hence we have  $v_{4,-1}\leq 2$.

In the proof we will always apply a setting such that, for some integer $j\geq 0$,
$$M_4\geq jF+S_{4,-j}$$
for a moving divisor $S_{4,-j}$ with $h^0(F,S_{4,-j}|_F)\geq 2$. Modulo further birational modifications, we may and do assume that $|S_{4,-j}|$ is base point free.
\medskip

{\bf Case 1.}  $u_{4,0}\ge 7$.

Since $v_{4,0}\leq 3$ and $v_{4,-1}\leq 2$, we have $h^0(F, M_4|_F-2C)\geq 2$.  There is a moving curve $C_{-2}$ such that $M_4|_F\geq 2C+C_{-2}$. When $|C|$ and $|C_{-2}|$ are composed of the same pencil, we get $M_4|_F\geq 3C$ and so $\beta\geq \frac{3}{4}$. Lemma \ref{lem:beta2} implies the birationality of $\varphi_{6,X}$. When $|C|$ and $|C_{-2}|$ are not composed of the same pencil,  \cite[Proposition 3.7]{CHP}(iii)
implies that we have $\xi\ge\frac{4}{5}$  ($n=4$, $m_1=4$, $j=2$, $\delta_1=2$, $\xi\ge\frac{2}{3}$, $\mu=1$). Therefore $\varphi_{6,X}$ is birational by \cite[Proposition 3.7]{CHP}(i) ($m_1=4$, $j=2$, $\delta_1=2$, $\xi\ge\frac{4}{5}$, $\mu=1$).
\medskip

{\bf Case 2.} $u_{4,0}=6$.

The argument is organised according to the value of $v_{4,0}$.
\medskip

\noindent{Subcase 2.1}.  $u_{4,0}=6$ and $v_{4,0}\leq 2$.

We have $M_4|_F\ge 2C+C_{-2}$, where $h^0(C_{-2})=h^0(M_4|_F-2C)\ge 2$. If $|C_{-2}|$ and $|C|$ are composed of the same pencil, then we have $\beta\ge\frac{3}{4}$ and $\varphi_{6,X}$ is birational by Lemma \ref{lem:beta2}. If $|C_{-2}|$ and $|C|$ are not composed of the same pencil, we have $\xi\ge\frac{4}{5}$ by  \cite[Proposition 3.7]{CHP}(iii) ($n=4$, $m_1=4$, $j=2$, $\delta_1=2$, $\xi\ge\frac{2}{3}$, $\mu=1$). Then $\varphi_{6,X}$ is birational by \cite[Proposition 3.7]{CHP}(i) ($m_1=4$, $j=2$, $\delta_1=2$, $\xi\ge\frac{4}{5}$, $\mu=1$).
\medskip

\noindent{Subcase 2.2}. $u_{4,0}=6$, $u_{4,-1}\ge 4$ and  $v_{4,0}=3$.

Since $v_{4,0}=3$, we have $\xi=1$. Clearly we have
$h^0(F, S_{4,-1}|_F)\geq 4$ by assumption.

If $|S_{4,-1}|_F|$ and $|C|$ are not composed of the same pencil and $(S_{4,-1}|_F\cdot C)\ge 4$. We have $$(\pi^*(K_X)|_F)^2\ge\frac{1}{5}(\xi+\frac{1}{2}(S_{4,-1}|_F\cdot C))\ge\frac{3}{5}>\frac{1}{2}$$
by Theorem \ref{k2}(iii). So $\varphi_{6,X}$ is birational by Lemma \ref{lem:restric vol}.

If $|S_{4,-1}|_F|$ and $|C|$ are not composed of the same pencil and $(S_{4,-1}|_F\cdot C)\le 3$. Then we have $$S_{4,-1}|_F\ge C+C_{-1},$$ where $C_{-1}$ is a moving curve on $F$. If $|C_{-1}|$ and $|C|$ are not composed of the same pencil, noting that
$$(\pi^*(K_X)|_F\cdot S_{4,-1}|_F)\geq \xi+\frac{1}{2}(C\cdot C_{-1})\geq 2,$$
we still have $(\pi^*(K_X)|_F)^2\geq \frac{3}{5}$ by Theorem \ref{k2}(iii) and above inequality. Hence  $\varphi_{6,X}$ is birational by Lemma \ref{lem:restric vol}. If $|C_{-1}|$ and $|C|$ are  composed of the same pencil. We have $\beta\ge\frac{3}{5}$ by  \cite[Proposition 3.5]{CHP} ($n_1=4$, $j_1=1$, $l_1=2$). Since $$(\pi^*(K_X)|_F)^2\ge\beta\cdot\xi \ge \frac{3}{5},$$ $\varphi_{6,X}$ is birational by Lemma \ref{lem:restric vol}.

If $|S_{4,-1}|_F|$ and $|C|$ are composed of the same pencil. We have $S_{4,-1}|_F\ge 3C$.  By  \cite[Proposition 3.5]{CHP}, we get $\beta\ge\frac{4}{5}$ and so $(\pi^*(K_X)|_F)^2\geq \frac{4}{5}$. Hence  $\varphi_{6,X}$ is birational by Lemma \ref{lem:beta2}.
\medskip

\noindent{Subcase 2.3}. $u_{4,0}=6$, $u_{4,-1}\le 3$, $u_{4,-2}\le 3$, $u_{4,-3}\ge 2$ and $v_{4,0}=3$.

$v_{4,0}=3$ implies that $\xi=1$.

If $|S_{4,-3}|_F|$ and $|C|$ are not composed of the same pencil, we have
$(\pi^*(K_X)|_F)^2\ge\frac{4}{7}>\frac{1}{2}$ by Theorem \ref{k2}(iii) and $\varphi_{6,X}$ is birational by Lemma \ref{lem:restric vol}.

If $|S_{4,-3}|_F|$ and $|C|$ are composed of the same pencil, we have $\beta\ge\frac{4}{7}$ by  \cite[Proposition 3.5]{CHP}.
Similarly, one has $(\pi^*(K_X)|_F)^2\ge\frac{4}{7}>\frac{1}{2}$. Hence $\varphi_{6,X}$ is birational by Lemma \ref{lem:restric vol}.
\medskip

\noindent{Subcase 2.4}. $u_{4,0}=6$, $u_{4,-1}\le 3$, $u_{4,-2}\le 3$, $u_{4,-3}=1$, $v_{4,0}=3$ and $P_4(X)\ge 15$.

$v_{4,0}=3$ implies that $\xi=1$.

As $|M_4-3F|$ and $|F|$ are composed of the same pencil and $h^0(M_4-3F)\geq 3$, we have $\mu\geq \frac{5}{4}$ and so $\beta\ge\frac{5}{9}$ by  \cite[Proposition 3.5]{CHP}. Hence $\varphi_{6,X}$ is birational by Lemma \ref{lem:restric vol} since $(\pi^*(K_X)|_F)^2\ge\frac{5}{9}>\frac{1}{2}$.
\medskip

%{\bf Case 3.} $u_{4,0}\le 5$ and $u_{4,-1}=5$.

%If $|S_{4,-1}|_F|$ and $|C|$ are not composed of the same pencil and $(S_{4,-1}|_F\cdot C)\geq 4$. We must have $\xi=1$. Since $(\pi^*(K_X)|_F)^2\ge\frac{3}{5}>\frac{1}{2}$ by Theorem \ref{k2}(iii), $\varphi_{6,X}$ is birational by Lemma \ref{lem:restric vol}.

%If $|S_{4,-1}|_F|$ and $|C|$ are not composed of the same pencil and $(S_{4,-1}\cdot C)\le 3$. We have $S_{4,-1}|_F\ge C+C_{-1}$ where $C_{-1}$ is a moving curve with $h^0(F,C_{-1})\ge 3$. If $|C_{-1}|$ and $|C|$ are composed of the same pencil, we have $\beta\ge\frac{4}{5}>\frac{2}{3}$ by  \cite[Proposition 3.5]{CHP} and hence $\varphi_{6,X}$ is birational by Lemma \ref{lem:beta2}. If $|C_{-1}|$ and $|C|$ are not composed of the same pencil, then $\varphi_{6,X}$ is birational by Theorem \ref{k3}.

%If $|S_{4,-1}|_F|$ and $|C|$ are composed of the same pencil, then $S_{4,-1}|_F\ge 4C$. We get $\beta\ge 1$ and $\varphi_{6,X}$ is birational by Lemma \ref{lem:beta2}.
%\medskip

{\bf Case 3.} $u_{4,0}\le 5$ and $u_{4,-1}\geq 4$.

If $|S_{4,-1}|_F|$ and $|C|$ are composed of the same pencil, then $\beta \ge \frac{4}{5}$ by  \cite[Proposition 3.5]{CHP}. Hence $\varphi_{6,X}$ is birational by Lemma \ref{lem:beta2}.

Assume that $|S_{4,-1}|_F|$ and $|C|$ are not composed of the same pencil. For the case  $(S_{4,-1}|_F\cdot C)=4$, %$\xi =1$ and $(\pi^*(K_X)|_F)^2\ge\frac{4}{5}>\frac{2}{3}$ by Theorem \ref{k2}(iii).
$\varphi_{6,X}$ is birational by Theorem \ref{k2}(ii) ($m_1=4$, $j=1$, $\tilde{\delta}=4$, $\beta=\frac{1}{2}$, $\mu=1$).
 For the case  $(S_{4,-1}|_F\cdot C)\le 3$, we have $S_{4,-1}|_F\ge C+C_{-1}$ where $C_{-1}$ is a moving curve. When $|C_{-1}|$ and $|C|$ are not composed of the same pencil, then
$\varphi_{6,X}$ is birational by Theorem \ref{k3}(i.2) ($m_1=4$, $j_1=j_2=1$, $\mu=1$, $\beta=\frac{1}{2}$, $\delta_2=2$).

When $|C_{-1}|$ and $|C|$ are composed of the same pencil, we have $\beta \geq \frac{3}{5}$ by  \cite[Proposition 3.5]{CHP}. Since $\alpha(7)>2$, we have $\xi\ge\frac{5}{7}$.
%{\color{red}{
%then Theorem \ref{kf}(ii.1) implies $\xi\geq \frac{5}{7}$ and $\varphi_{6,X}$ is birational by Theorem \ref{kf}(ii.2).}}
\medskip

 \noindent{Subcase 3.1.} $u_{4,0}\le 5$,  $u_{4,-1}=4$ and $u_{4,-2}=4$.

Since $|S_{4,-2}|_F|=|S_{4,-1}|_F|$, we have $S_{4,-2}|_F\ge 2C$. Hence $\varphi_{6,X}$ is birational by Theorem \ref{kf} (i.2)($\xi=\frac{5}{7}$, $m_1=4$, $j_1=j_2=2$, $\beta=\frac{3}{5}$).
\medskip

\noindent{Subcase 3.2.} $u_{4,0}\le 5$,  $u_{4,-1}=4$, $u_{4,-2}\le 3$ and $u_{4,-3}\ge 2$.

If $|S_{4,-3}|_F|$ and $|C|$ are not composed of the same pencil, $\varphi_{6,X}$ is birational by Theorem \ref{k2} (i) ($m_1=4$, $j=3$, $\xi=\frac{5}{7}$, $\tilde{\delta}=2$, $\beta=\frac{3}{5}$). If $|S_{4,-3}|_F|$ and $|C|$ are composed of the same pencil, we have $\xi\ge\frac{4}{5}$ by Theorem \ref{kf}(i.1)($n=4$, $m_1=4$, $j_1=3$, $j_2=1$, $\beta=\frac{3}{5}$, $\xi=\frac{5}{7}$).  Hence $\varphi_{6,X}$ is birational by Theorem \ref{kf}(i.2) ($m_1=4$, $j_1=3$, $j_2=1$, $\xi=\frac{4}{5}$, $\beta=\frac{3}{5}$).
\medskip

\noindent{Subcase 3.3.} $u_{4,0}\le 5$,  $u_{4,-1}=4$, $u_{4,-2}\le 3$, $u_{4,-3}=1$ and $P_{4}(X)\ge 15$.

We have $\zeta(4)\ge5$ by our assumption. Recall that we have $\beta\ge\frac{3}{5}$ in this case. We get $\xi\ge\frac{4}{5}$ since $\alpha(5)\ge\frac{46}{45}>1$. Hence $\varphi_{6,X}$ is birational as $\alpha(6)\ge\frac{152}{75}>2$.
\medskip

{\bf Case 4}.  $u_{4,0}\le 5$, $u_{4,-1}\leq 3$ and $u_{4,-3}\geq 3$.
%{\bf Case 5.} $u_{4,0}\le 5$ and $u_{4,-3}\ge 3$.

If $|S_{4,-3}|_F|$ and $|C|$ are composed of the same pencil, we have $\beta\ge \frac{5}{7}>\frac{2}{3}$ by  \cite[Proposition 3.5]{CHP}. So $\varphi_{6,X}$ is birational by Lemma \ref{lem:beta2}.

Assume that $|S_{4,-3}|_F|$ and $|C|$ are not composed of the same pencil. When $(S_{4,-3}|_F\cdot C)\ge 3$,
we have $\xi \geq \frac{6}{7}$ by Theorem \ref{k2}(iii). Thus Theorem \ref{k2}(i) ($\beta=\frac{1}{2}$) implies the birationality of $\varphi_{6,X}$.   When $(S_{4,-3}|_F\cdot C)\le 2$, we have $S_{4,-3}|_F\ge C$ and so $\beta\ge\frac{4}{7}$ by  \cite[Proposition 3.5]{CHP}.  Also we have $\xi \geq \frac{5}{7}$ by Theorem \ref{k2}(iii) ($j=3$, $m_1=4$).
Hence $\varphi_{6,X}$ is birational by Theorem \ref{k2}(i) ($m_1=4$, $j=3$, $\tilde{\delta}=2$, $\xi=\frac{5}{7}$, $\beta=\frac{4}{7}$).
\medskip

{\bf Case 5.} $u_{4,0}\le 5$, $u_{4,-1}\le 3$, $u_{4,-3}\le 2$ and $u_{4,-4}\geq 2$.

If $|S_{4,-4}|_F|$ and $|C|$ are composed of the same pencil, we get $\beta\ge\frac{5}{8}$ by  \cite[Proposition 3.5]{CHP}.
By Theorem \ref{kf} (i.1), we have $\xi\geq \frac{4}{5}$ ($m_1=4$, $j_1=4$, $j_2=1$, $\beta=\frac{5}{8}$, $\xi=\frac{2}{3}$, $n=4$). Then $\varphi_{6,X}$ is birational by Theorem \ref{kf} (i.2) ($m_1=4$, $j_1=4$, $j_2=1$, $\beta=\frac{5}{8}$, $\xi=\frac{4}{5}$).

If $|S_{4,-4}|_F|$ and $|C|$ are not composed of the same pencil,  we have $\xi\geq \frac{3}{4}$ by Theorem \ref{k2}(iii). Furthermore one gets $\xi\geq \frac{7}{9}$ by Theorem \ref{k2}(iv) ($n=8$). {}Finally  $\varphi_{6,X}$ is birational by Theorem \ref{k2}(i) ($m_1=4$, $j=4$, $\beta=\frac{1}{2}$, $\xi=\frac{7}{9}$, $\tilde{\delta}=2$).
\medskip

{\bf Case 6.} $u_{4,0}\le 5$, $u_{4,-1}\le 3$, $u_{4,-3}\leq 2$, $u_{4,-4}=1$ and $P_4(X)\ge15$.

In any case, the assumption implies that $\mu\geq \frac{5}{4}$.

If $u_{4,0}\le 4$, we get $h^0(M_4-4F)\ge 3$.  One has $M_4\geq 6F$ which means $\mu\geq \frac{3}{2}>\frac{4}{3}$. Hence  $\varphi_{6,X}$ is birational by Lemma \ref{lem:zeta}.

If $u_{4,0}=5$, then either $(M_4|_F\cdot C)=4$ or $M_4|_F\ge C+C_{-1}$, where $C_{-1}$ is a moving curve with $h^0(C_{-1})\geq 3$. We have $\xi=1$ in first case and, since $\beta\geq \frac{5}{9}$ by \eqref{cri},
one has
 $(\pi^*(K_X)|_F)^2\geq \frac{5}{9}>\frac{1}{2}$.  Hence $\varphi_{6,X}$ is birational by Lemma \ref{lem:restric vol}. Now turn to the later case. When $|C_{-1}|$ and $|C|$ are composed of the same pencil, we have $\beta\geq \frac{3}{4}$ and $\varphi_{6,X}$ is birational by Lemma \ref{lem:beta2}. If $|C_{-1}|$ and $|C|$ are not composed of the same pencil, $\varphi_{6,X}$ is birational by \cite[Proposition 3.7]{CHP} ($j=1$, $\delta_1=2$, $\mu=\frac{5}{4}$).
 %{\bf Case 8.} $u_{4,0}\le 5$, $u_{4,-1}\le 3$, $u_{4,-2}\le 3$, $u_{4,-3}=1$ and $P_4(X)\ge 15$.
%We have $\zeta\ge\frac{3}{2}$ by our assumption and $\varphi_{6,X}$ is birational by Lemma \ref{lem:beta2}.

Now suppose $P_4(X)=14$.

We first assume that $\varphi_{6,X}$ is not birational. By the arguments in {\bf Case 1}-{\bf Case 6}, one of the following holds.
\begin{enumerate}
\item[(a)] $u_{4,0}=6$, $u_{4,-1}\le 3$, $u_{4,-2}\le 3$, $u_{4,-3}=1$, $v_{4,0}=3$;
\item[(b)] $u_{4,0}\le 5$, $u_{4,-1}\le 3$, $u_{4,-2}\le 3$ and $u_{4,-3}= 2$;
\item[(c)] $u_{4,0}\le 5$, $u_{4,-1}\le 3$, $u_{4,-2}\le 3$ and $u_{4,-3}=1$.
\end{enumerate}
 For  (a), we have $\xi=1$ since $v_{4,0}=3$. Lemma \ref{lem:restric vol} implies that we have $(\pi^*(K_X)|_F)^2=\frac{1}{2}$.  Thus (2.2) holds.

For (b). If $|S_{4,-3}|_F|$ and $|C|$ are not composed of the same pencil, we have $(S_{4,-3}|_F\cdot C)\ge 2$. By Theorem \ref{k2} (iv), we have $\xi\ge\frac{4}{5}$ ($n=4$). Theorem \ref{k2} (i) implies that $\varphi_{6,X}$ is birational when $\xi>\frac{4}{5}$ ($\tilde{\delta}=2$, $j=3$, $\xi>\frac{4}{5}$, $\beta=\frac{1}{2}$). So $\xi=\frac{4}{5}$ holds.

If $|S_{4,-3}|_F|$ and $|C|$ are composed of the same pencil. By  \cite[Proposition 3.5]{CHP}, we have $\beta\ge\frac{4}{7}$. Since $$\alpha(7)>(7-1-1-\frac{1}{\beta})\cdot \xi \ge\frac{13}{6}>2,$$ we have $\xi\ge\frac{5}{7}$. By Theorem \ref{kf} (i.1), we have $\xi\ge\frac{4}{5}$ ($m_1=4$, $j_1=3$, $j_2=1$, $\beta=\frac{4}{7}$, $\xi=\frac{5}{7}$). Theorem \ref{kf} (i.2) implies that $\varphi_{6,X}$ is birational when $\xi>\frac{4}{5}$ ($j_1=3$, $j_2=1$, $\xi>\frac{4}{5}$, $\beta=\frac{4}{7}$). Thus we have $\xi=\frac{4}{5}$.

We are left to treat (c). We claim that $u_{4,0}=5$. Otherwise, one has $u_{4,0}\le 4$. Thus we have  $h^0(M_4-3F)\ge 4$. Then $\varphi_{6,X}$ is birational by Lemma \ref{lem:Pm1 bir} (ii) ($m_1=4$, $j=3$), which contradicts to our assumption. So we have $u_{4,0}=5$. By our assumption, we have $h^0(M_4-3F)\ge 3$. Thus one has $\mu\ge\frac{5}{4}$ and $\beta\ge\frac{5}{9}$ by  \cite[Proposition 3.5]{CHP}.

If $v_{4,0}\le 2$, one has $M_4|_F\ge C+C_{-1}$, where $C_{-1}$ is a moving curve satisfying $h^0(F, C_{-1})\ge 3$. If $|C_{-1}|$ and $|C|$ are composed of the same pencil, one gets $\beta\ge\frac{3}{4}$. Lemma \ref{lem:beta2} implies that $\varphi_{6,X}$ is birational, which contradicts to our assumption. Thus  we have $(C_{-1}\cdot C)\ge 2$. \cite[Proposition 3.7]{CHP} (i) and (ii) implies that $\varphi_{6,X}$ is birational ($m_1=4$, $j=1$, $\mu=\frac{5}{4}$, $\beta=\frac{5}{9}$), which is a contradiction. So we have $v_{4,0}=3$. In particular, one has $\xi=1$. Lemma \ref{lem:restric vol} implies that we have $(\pi^*(K_X)|_F)^2=\frac{1}{2}$.

Now we consider the other direction. Lemma \ref{lem:charac K1} implies that we\ only need to consider the case where $\xi=1$ and $(\pi^*(K_X)|_F)^2=\frac{1}{2}$. Observing that $\beta>\frac{1}{2}$ implies that $(\pi^*(K_X)|_F)^2>\frac{1}{2}$. Thus we get $\mu=1$ and $\beta=\frac{1}{2}$. By the argument in {\bf Case 1}-{\bf Case 6}, one of the following holds:
\begin{enumerate}
\item [(i)] $u_{4,0}=6$, $u_{4,-1}\le 3$, $u_{4,-2}\le 3$, $u_{4,-3}=1$ and $v_{4,0}=3$;
\item [(ii)] $u_{4,0}\le 5$, $u_{4,-1}\ge 4$, $S_{4,-1}|_F\ge C+C_{-1}$, where $C_{-1}$ is a moving curve. Moreover, $|C_{-1}|$ and $|C|$ are not composed of the same pencil;
\item [(iii)] $u_{4,0}\le 5$, $u_{4,-1}\le 3$, $u_{4,-3}=2$;
\item [(iv)] $u_{4,0}\le 5$, $u_{4,-1}\le 3$,  $u_{4,-3}=1$.
\end{enumerate}

We first consider (i). Our assumption gives $M_4|_F\ge C+C_{-1}$, where $C_{-1}$ is a moving curve on $F$  satisfying $h^0(F, C_{-1})\ge 3$. Since $q(F)=0$, $\xi=1$ and $(\pi^*(K_X)|_F)^2=\frac{1}{2}$, $|C_{-1}|$ is not composed of pencil and we have $(\pi^*(K_X)|_F\cdot C_{-1})=1$.  We may and do assume that $|C_{-1}|$ is base point free. Take a general member $C_{-1} \in |C_{-1}|$. One has $g(C_{-1})\ge 3$ and $|C|_{C_{-1}}|$ is $g^1_{2}$. By Kawamata-Viehweg vanishing theorem, we have
\begin{align*}
|6K_{X'}||_F\lsgeq |K_{X'}+M_4+F||_F\lsgeq |K_F+C+C_{-1}|.
\end{align*}
By Ramanujam vanishing theorem, one has $h^{1}(F, K_F+C)=0$. Thus we have $$|K_F+C+C_{-1}||_{C_{-1}}=|K_{C_{-1}}+C|_{C_{-1}}|.$$ So $$|M_6||_{C_{-1}}\lsgeq |K_{C_{-1}}+C|_{C_{-1}}|.$$ Since $$\mathrm{deg}(K_{C_{-1}}+C|_{C_{-1}})\ge 6 \ \ \ \ \text{and}\ \ \ \ (M_6\cdot C_{-1})\le (6\pi^*(K_X)|_F\cdot C_{-1})=6,$$ we have $$|M_6||_{C_{-1}}=|K_{C_{-1}}+C|_{C_{-1}}|.$$ Since $|C|_{C_{-1}}|$ is $g^1_{2}$, $\varphi_{6,X}$ is non-birational.

For (ii). By Theorem \ref{k2} (iii) ($j=1$), we have $$\pi^*(K_X)|_F\ge \frac{1}{5}C+\frac{1}{5}S_{4,-1}|_F\ge \frac{2}{5} C+\frac{1}{5}C_{-1}.$$
Since $\xi=1$ and $\beta\ge\frac{1}{2}$, we have $$(\pi^*(K_X)|_F)^2\ge\frac{2}{5}+\frac{1}{5}\cdot \frac{1}{2}(C\cdot C_{-1})\ge\frac{3}{5},$$ which contradicts to our assumption.

For (iii).  If $|S_{4,-3}|_F|$ and $|C|$ are composed of the same pencil, we have $\beta\ge\frac{4}{7}$ by  \cite[Proposition 3.5]{CHP}, which contradicts to our assumption. Thus $|S_{4,-3}|_F|$ and $|C|$ are not composed of the same pencil. In particular, we have $(S_{4,-3}|_F\cdot C)\ge 2$. By Theorem \ref{k2} (iii) ($j=3$), we have $$\pi^*(K_X)|_F\ge\frac{3}{7}C+\frac{1}{7}S_{4,-3}|_F.$$ One can gets $(\pi^*(K_X)|_F)^2\ge\frac{4}{7}$, which is a contradiction.

We are left to treat (iv). By our assumption, we have  $\mu\ge\frac{5}{4}$, which is a contradiction.
\end{proof}

\begin{prop}\label{prop:P5}
Under the same condition as that of Lemma \ref{lem:beta}, then
\begin{enumerate}
\item when $P_5(X)\ge 24$, then $\varphi_{6,X}$ is birational;
\item when $22\le P_5(X)\le 23$, $\varphi_{6,X}$ is non-birational if and only if $\xi=\frac{4}{5}$.
 \end{enumerate}
\end{prop}

\begin{proof}
Set $m_1=5$. By Lemma \ref{lem:Pm1 bir}, we may and do assume that $u_{5,0}\le h^0(5K_F)-1=12$.

 By \cite[Proposition 3.4]{CHP}, we know that $v_{5,0}\le 4$. When $\mathrm{dim}\psi_{5,0}(U_{5,0})=4$, we have $\mathrm{deg}(M_5|_C)=5$, which implies that $\xi=1$.  By Riemann-Roch formula, we have $h^0(C, M_5|_C)=4$. So $|M_5||_C$ is the complete linear system $|{M_5}|_C|=|K_C+D_1|$ with $\mathrm{deg}(D_1)=3$. Thus $\varphi_{5,X}$ is birational which implies that $\varphi_{6,X}$ is birational. So we may assume that $\mathrm{dim}\psi_{5,0}(U_{5,0})\le 3$.
%Since $(S_{5,0}|_F\cdot C)\le (5\sigma^*(K_{F_0})\cdot C)=5$, we have $v_{5,0}\le 4$. If $v_{5,0}=4$, then $|S_{5,0}||_F|_C=|S_{5,0}|_F|_C|$. Since $\mathrm{deg}(S_{5,0}|_F|_C)=5$, $\varphi_{|S_{5,0}|_F|_C|}$ is birational. It is easy to see that $\varphi_{5,X}$ is birational under this assumption. $\varphi_{6,X}$ is birational because $p_g(X)>0$.

%If $v_{5,-1}\ge 3$, then we have $\xi\ge\frac{4}{5}$. $\varphi_{6,X}$ is birational by Lemma \ref{lem:beta2}.
%So we make assumption that $v_{5,0}\le 3$ and $v_{5,-1}\le 2$ throughout this proof.
Suppose $v_{5,-1}\ge 3$. Then $M_5|_F\ge C+C_{-1}$ for some moving divisor $C_{-1}$ satisfying $h^0(C, C_{-1}|_C)\ge 3$. By Riemann-Roch formula, we have $(C_{-1}\cdot C)\ge 4$. Then $\varphi_{6,X}$ is birational by \cite[Proposition 3.7]{CHP} (ii) ($\mu=1$, $m_1=5$, $j=1$, $\delta_1=4$, $\beta=\frac{1}{2}$).  From now on, we may assume that $v_{5,-1}\le 2$.
\medskip

{\bf Case 1.} $u_{5,0}\ge 8$.

Since $\mathrm{dim}\psi_{5,0}(U_{5,0})\le 3$, we have $h^0(M_5|_F-C)\ge 5$. Because $v_{5,-1}\le 2$, we have $M_5|_F\ge 2C+C_{-2}$, where $C_{-2}$ is a moving curve on $F$ satisfying $h^0(C_{-2})\ge 3$. If $|C_{-2}|$ and $|C|$ are composed of the same pencil, we have $\beta\ge\frac{4}{5}$ and $\varphi_{6,X}$ is birational by Lemma \ref{lem:beta2}.  If $|C_{-2}|$ and $|C|$ are not composed of the same pencil, we have $(C_{-2}\cdot C)\ge 2$. We have $$(C_{-2}\cdot C)\le (C_{-2}\cdot \sigma^*(K_{F_0}))\le ((M_5|_F-2C)\cdot \sigma^*(K_{F_0}))\le 3.$$ By \cite[Proposition 3.7]{CHP}(iii), we have $\xi\ge\frac{4}{5}$ ($n=4$, $\mu=1$, $m_1=5$, $j=2$, $\xi(1,|C|)=\frac{2}{3}$). Then $\varphi_{6,X}$ is birational by \cite[Proposition 3.7]{CHP} (i) ($\mu=1$, $m_1=5$, $j=2$, $\delta_1=2$, $\xi=\frac{4}{5}$).
\medskip

{\bf Case 2.} $u_{5,0}=7$.

If $\mathrm{dim}\psi_{5,0}(U_{5,0})\le 2$, we have $h^0(M_5|_F-C)\ge 5$. The same argument as in {Case 1} implies that $\varphi_{6,X}$ is birational. So we may assume that $\mathrm{dim} \psi_{5,0}(U_{5,0})=3$, which implies that $\xi\ge\frac{4}{5}$.

Since $v_{5,-1}\le 2$, we have $M_5|_F\ge 2C+C_{-2}$, where $C_{-2}$ is a moving curve satisfying $h^0(C_{-2})\ge 2$. If $|C_{-2}|$ and $|C|$ are not composed of the same pencil, we have $2\le(C_{-2}\cdot C)\le ((M_5|_F-2C)\cdot \sigma^*(K_{F_0}))\le 3$. \cite[Proposition 3.7]{CHP} (i) implies that $\varphi_{6,X}$ is birational ($\mu=1$, $m_1=5$, $j=2$, $\delta_1=2$, $\xi=\frac{4}{5}$).

 Suppose $|C_{-2}|$ and $|C|$ are composed of the same pencil. We get $\beta\ge\frac{3}{5}$ by our assumption. So we may assume that $\xi\ge\frac{4}{5}$ and $\beta\ge\frac{3}{5}$ in this case.
\medskip

Subcase 2.1. $u_{5,-1}\ge 6$.

If $\mathrm{dim} \psi_{5,0}(U_{5,-1})=3$, we have $(S_{5,-1}|_F\cdot C)\ge 4$. By Theorem \ref{KaE}, we have $$|4K_{X'}||_F\lsgeq |2(K_X'+F)||_F\lsgeq |2\sigma^*(K_{F_0})|.$$ \cite[Lemma 3.1]{CHP} implies that $M_4$ is a big divisor. So $S_{5,-1}$ is nef and big. Kawamata-Viehweg vanishing theorem yields $|6K_{X'}||_F\lsgeq |K_F+S_{5,-1}|_F|$. Thus $M_6|_F\ge C+S_{5,-1}|_F$. Then $\varphi_{6,X}$ is birational by \cite[Proposition 3.7]{CHP} (ii) ($\mu=1$, $m_1=5$, $j=1$, $\delta_1=4$, $\beta=\frac{3}{5}$, $\xi=\frac{4}{5}$).

If $\mathrm{dim} \psi_{5,0}(U_{5,-1})\le 2$, we have $S_{5,-1}|_F\ge 2C+C'$ where $C'$ is a moving divisor on $F$. If $|C'|$ and $|C|$ are not composed of the same pencil, we have $2\le (C'\cdot C)\le (5\pi^*(K_X)|_F-2C)\cdot \sigma^*(K_{F_0})\le 3$. Theorem \ref{k3} (ii) implies that $\varphi_{6,X}$ is birational ($\mu=1$, $m_1=5$, $j_1=1$, $j_2=2$, $\delta_2=2$, $\xi=\frac{4}{5}$). If $|C'|$ and $|C|$ are  composed of the same pencil, we have $S_{5,-1}|_F\ge 3C$. By Theorem \ref{kf} (ii.2), $\varphi_{6,X}$ is birational ($\mu=1$, $m_1=5$, $j_1=1$, $j_2=3$, $\xi=\frac{4}{5}$).
\medskip

Subcase 2.2. $u_{5,-1}\le 5$,  $u_{5,-2}=5$.

If $(S_{5,-2}|_F\cdot C)=5$, we have $\xi=1$. Since $\beta\ge\frac{3}{5}$, $\varphi_{6,X}$ is birational by Lemma \ref{lem:restric vol}. So we may assume $(S_{5,-2}|_F\cdot C)\le 4$.

If $(S_{5,-2}|_F\cdot C)=4$,  $\varphi_{6,X}$ is birational by Theorem \ref{k2} (i) ($m_1=5$, $j=2$, $\tilde{\delta}=4$, $\beta=\frac{3}{5}$, $\xi=\frac{4}{5}$).

We are left to treat the case when  $(S_{5,-1}|_F\cdot C)\le 3$. We have $S_{5,-2}|_F\ge C+C'$, where $C'$ is a moving curve satisfying $h^0(F, C')\ge 3$. If $|C'|$ and $|C|$ are composed of the same pencil, we have $\beta\ge\frac{5}{7}>\frac{2}{3}$ by  \cite[Proposition 3.5]{CHP}. By Lemma \ref{lem:beta2}, $\varphi_{6,X}$ is birational. If $|C'|$ and $|C|$ are not composed of the same pencil, $\varphi_{6,X}$ is birational by Theorem \ref{k3} (i.2) ($m_1=5$, $j_1=2$, $j_2=1$, $\delta_2=2$, $\beta=\frac{3}{5}$, $\xi=\frac{4}{5}$ ).
\medskip

Subcase 2.3. $u_{5,-1}\le 5$, $u_{5,-2}\le 4$ and $u_{5,-3}\ge 4$.

If $(S_{5,-3}|_F\cdot C)\ge 4$, $\varphi_{6,X}$ is birational by Theorem \ref{k2} (i) ($m_1=5$, $j=3$, $\tilde{\delta}=4$, $\beta=\frac{3}{5}$, $\xi=\frac{4}{5}$).

 If $(S_{5,-3}|_F\cdot C)\le 3$, we have $S_{5,-3}|_F\ge C+C'$, where $C'$ is a moving curve. If $|C'|$ and $|C|$ are composed of the same pencil, we have $\beta\ge\frac{5}{8}$ and $\varphi_{6,X}$ is birational by Theorem \ref{kf} (i.2) ($m_1=5$, $j_1=3$, $j_2=2$,  $\beta=\frac{5}{8}$, $\xi=\frac{4}{5}$). If $|C'|$ and $|C|$ are not composed of the same pencil, we have $(C'\cdot C)\ge 2$ and $\varphi_{6,X}$ is birational by Theorem \ref{k3} (i.2) ($m_1=5$, $j_1=3$, $j_2=1$, $\delta_2=2$, $\beta=\frac{3}{5}$, $\xi=\frac{4}{5}$).
\medskip

Subcase 2.4. $u_{5,-1}\le 5$, $u_{5,-2}\le 4$, $u_{5,-3}\le 3$ and  $u_{5,-4}\ge 3$.

If $|S_{5,-4}|_F|$ and $|C|$ are composed of the same pencil, we have $\beta \ge\frac{2}{3}$ by  \cite[Proposition 3.5]{CHP}. Then $\varphi_{6,X}$ is birational by Theorem \ref{kf} (i.2) ($m_1=5$, $j_1=4$, $j_2=2$, $\beta=\frac{2}{3}$, $\xi=\frac{4}{5}$).

If $|S_{5,-4}|_F|$ and $|C|$ are not composed of the same pencil, we have $2\le (S_{5,-4}|_F\cdot C)\le 5$. Then $\varphi_{6,X}$ is birational by Theorem \ref{k2} (i) ($m_1=5$, $j=4$, $\tilde{\delta}=2$, $\beta=\frac{3}{5}$, $\xi=\frac{4}{5}$).
\medskip

Subcase 2.5. $u_{5,-1}\le 5$, $u_{5,-2}\le 4$, $u_{5,-3}\le 3$, $u_{5,-4}\le 2$ and $u_{5,-5}\ge 2$.

If $|S_{5,-5}|_F|$ and $|C|$ are composed of the same pencil, $\varphi_{6,X}$ is birational by Theorem \ref{kf} (i.2) ($m_1=5$, $j_1=5$, $j_2=1$, $\beta=\frac{3}{5}$, $\xi=\frac{4}{5}$).

If $|S_{5,-5}|_F|$ and $|C|$ are not composed of the same pencil, $\varphi_{6,X}$ is birational by Theorem \ref{k2} (i) ($m_1=5$, $j=5$, $\tilde{\delta}=2$, $\beta=\frac{3}{5}$, $\xi=\frac{4}{5}$).
\medskip

Subcase 2.6. $u_{5,-1}\le 5$, $u_{5,-2}\le 4$, $u_{5,-3}\le 3$, $u_{5,-4}\le 2$, $u_{5,-5}=1$ and $P_5(X)\ge 24$.

Since $h^0(M_5-5F)\ge 3$, we have $\mu\ge\frac{7}{5}>\frac{4}{3}$.  Thus $\varphi_{6,X}$ is birational by Lemma \ref{lem:zeta}.
\medskip

{\bf Case 3.} $u_{5,0}\le 6$, $u_{5,-1}\le 6$, $u_{5,-2}\ge 5$.

 If $(S_{5,-2}|_F\cdot C)\ge 5$, $\varphi_{6,X}$ is birational by Theorem \ref{k2} (ii) ($m_1=5$, $j=2$, $\beta=\frac{1}{2}$, $\tilde{\delta}=5$).

 If $(S_{5,-2}|_F\cdot C)=4$, $\varphi_{6,X}$ is birational by Theorem \ref{k2} (i) ($m_1=5$, $j=2$, $\xi=\frac{4}{5}$, $\tilde{\delta}=4$, $\beta=\frac{1}{2}$).

 The remaining case is $(S_{5,-2}|_F\cdot C)\le 3$. By Riemann-Roch formula, we have $\mathrm{dim}\psi_{5,0}(U_{5,-2})\le 2$. So we have $S_{5,-2}|_F\ge 2C$ and $S_{5,-2}|_F\ge C+C'$, where $C'$ is a moving curve satisfying $h^0(F, C')\ge 3$. The former implies $\beta\ge\frac{4}{7}$ by  \cite[Proposition 3.5]{CHP}. If $|C'|$ and $|C|$ are composed of the same pencil, we have $\beta\ge\frac{5}{7}>\frac{2}{3}$. Lemma \ref{lem:beta2} implies that $\varphi_{6,X}$ is birational. If $|C'|$ and $|C|$ are not composed of the same pencil, we have $$2\le (C'\cdot C)\le (C'\cdot \sigma^*(K_{F_0}))\le ((S_{5,-2}|_F-C)\cdot \sigma^*(K_{F_0}))\le 4.$$ By Theorem \ref{k3} (i.2), $\varphi_{6,X}$ is birational ($\xi=\frac{2}{3}$, $\delta_2=2$, $j_1=2$, $j_2=1$, $\beta=\frac{4}{7}$).
\medskip

{\bf Case 4.} $u_{5,0}\le 6$, $u_{5,-1}\le 6$, $u_{5,-2}\le 4$, $u_{5,-3}\ge 4$.

If $(S_{5,-3}|_F\cdot C)\ge 4$, $\varphi_{6,X}$ is birational by Theorem \ref{k2} (i) ($\xi=\frac{4}{5}$, $\tilde{\delta}=4$, $j=3$, $\beta=\frac{1}{2}$, $m_1=5$).
So we may assume  $(S_{5,-3}|_F\cdot C)\le 3$.
By Riemann-Roch formula, we have $S_{5,-3}|_F\ge C+C'$ where $C'$ is a moving curve on $F$.

If $|C'|$ and $|C|$ are composed of the same pencil, we have $\beta\ge\frac{5}{8}$ by  \cite[Proposition 3.5]{CHP}. We have $\alpha(7)\ge \frac{34}{15}>2$, which implies $\xi\ge\frac{5}{7}$. Since $\alpha(8)\ge \frac{22}{7}>3$, we have $\xi\ge\frac{3}{4}$. So $\alpha(5)\ge\frac{21}{20}>1$ and  $\xi\ge\frac{4}{5}$ follows. By Theorem \ref{kf} (i.2),  $\varphi_{6, X}$ is birational ($\xi=\frac{4}{5}$,  $m_1=5$, $j_1=3$, $j_2=2$, $\beta=\frac{5}{8}$).

If $|C'|$ and $|C|$ are not composed of the same pencil, we have $2\le (C'\cdot C)\le 4$. By Theorem \ref{k3} (i.1), we have $\xi\ge\frac{5}{7}$ ($n=6$, $\xi=\frac{2}{3}$, $\delta_2=2$, $j_1=3$, $j_2=1$, $m_1=5$, $\beta=\frac{1}{2}$). So  $\varphi_{6,X}$ is birational by Theorem \ref{k3}(i.2) ($\xi=\frac{5}{7}$, $\delta_2=2$, $j_1=3$, $j_2=1$, $m_1=5$, $\beta=\frac{1}{2}$).
\medskip

{\bf Case 5.} $u_{5,0}\le 6$, $u_{5,-1}\le 6$, $u_{5,-2}\le 4$, $u_{5,-3}\le 3$, $u_{5,-4}\ge 3$.

If $(S_{5,-4}|_F\cdot C)\ge 4$, we have $\xi\ge\frac{4}{5}$. By Theorem \ref{k2} (i), $\varphi_{6,X}$ is birational ($\tilde{\delta}=4$, $\xi=\frac{4}{5}$, $j=4$, $m_1=5$, $\beta=\frac{1}{2}$). We may assume that $(S_{5,-4}|_F\cdot C)\le 3$. By Riemann-Roch formula, we have $S_{5,-4}|_F\ge C$. We get $\beta\ge\frac{5}{9}$ by  \cite[Proposition 3.5]{CHP}.

If $|S_{5,-4}|_F|$ and $|C|$ are composed of the same pencil, we have $\beta\ge\frac{2}{3}$ by  \cite[Proposition 3.5]{CHP}. By Theorem \ref{kf} (i.1), we have $\xi\ge\frac{4}{5}$ ($n=4$, $\xi=\frac{2}{3}$, $j_1=4$, $m_1=5$, $j_2=2$, $\beta=\frac{2}{3}$ ). Then $\varphi_{6,X}$ is birational by Theorem \ref{kf} (i.2) ($\xi=\frac{4}{5}$, $j_1=4$, $m_1=5$, $j_2=2$, $\beta=\frac{2}{3}$).

If $|S_{5,-4}|_F|$ and $|C|$ are not composed of the same pencil, we have $(S_{5,-4}|_F\cdot C)\ge 2$. By Theorem \ref{k2} (iv), we have $\xi\ge\frac{4}{5}$ ($n=4$, $m_1=5$, $j=4$, $\beta=\frac{5}{9}$, $\tilde{\delta}=2$). Hence $\varphi_{6,X}$ is birational by Theorem \ref{k2} (i) ($\xi=\frac{4}{5}$, $\tilde{\delta}=2$, $j=4$, $m_1=5$, $\beta=\frac{5}{9}$).
\medskip

{\bf Case 6.} $u_{5,0}\le 6$, $u_{5,-1}\le 6$, $u_{5,-2}\le 4$, $u_{5,-3}\le 3$, $u_{5,-4}\le 2$, $u_{5,-5}\ge 2$.

If $|S_{5,-5}|_F|$ and $|C|$ are composed of the same pencil, we have $\beta\ge\frac{3}{5}$ by  \cite[Proposition 3.5]{CHP}. By Theorem \ref{kf} (i.1), we have $\xi\ge\frac{4}{5}$ ($n=4$, $m_1=5$, $j_1=5$, $j_2=1$, $\beta=\frac{3}{5}$, $\xi=\frac{2}{3}$). By Theorem \ref{kf} (i.2), $\varphi_{6,X}$ is birational ($\xi=\frac{4}{5}$, $m_1=5$, $j_1=5$, $j_2=1$, $\beta=\frac{3}{5}$).

If $|S_{5,-5}|_F|$ and $|C|$ are not composed of the same pencil, we have $(S_{5,-5}|_F\cdot C)\ge 2$. By Theorem \ref{k2} (iv) , we have $\xi\ge\frac{4}{5}$. By Theorem \ref{k2} (i)$'$, $\varphi_{6,X}$ is birational ($m_1=5$, $j=5$, $\tilde{\delta}=2$, $\xi=\frac{4}{5}$).
\medskip

{\bf Case 7.}  $u_{5,0}\le 6$, $u_{5,-1}\le 6$, $u_{5,-2}\le 4$, $u_{5,-3}\le 3$, $u_{5,-4}\le 2$, $u_{5,-5}\le 1$ and $P_5(X)\ge 24$.

We have $h^0(M_5-5F)\ge 3$ by our assumption. Since $u_{5,-5}=1$, we have $\mu\ge\frac{7}{5}>\frac{4}{3}$ by  \cite[Proposition 3.5]{CHP}. Then $\varphi_{6,X}$ is birational by Lemma \ref{lem:zeta}.

Now we prove the second statement. Assume that $22\le P_5(X)\le 23$.

By Lemma \ref{lem:charac K1}, it suffices to consider the direction by assuming that $\varphi_{6,X}$ is not birational.  %Suppose that  $\xi\neq \frac{4}{5}$ .

 By the arguments in {\bf Case 1}$\sim${\bf Case 7}, it suffices to consider one of the following situations:
\begin{enumerate}
\item[(i)] $u_{5,-2}=4$;
\item[(ii)] $u_{5,0}=7$, $u_{5,-1}\le 5$, $u_{5,-2}\le 3$, $u_{5,-3}\le 3$, $u_{5,-4}\le 2$, $\mathrm{dim}\psi_{5,0}(U_{5,0})=3$, $\xi\ge\frac{4}{5}$, $\beta\ge\frac{3}{5}$;
\item[(iii)] $u_{5,0}\le 6$, $u_{5,-1}=6$, $u_{5,-2}\le 3$, $u_{5,-4}\le 2$, $u_{5,-5}=1$;
\item[(iv)] $u_{5,0}\le 6$, $u_{5,-1}\le 5$, $u_{5,-2}\le 3$,  $u_{5,-4}\le 2$, $u_{5,-5}=1$.
\end{enumerate}

We first consider (i). If $(S_{5,-2}|_F\cdot C)\ge 4$,  $\varphi_{6,X}$ is birational by Theorem \ref{k2} (i) and (ii) ($\tilde{\delta}\ge 4$, $m_1=5$, $j=2$, $\beta=\frac{1}{2}$), which contradicts to our assumption. So we have $(S_{5,-2}|_F\cdot C)\le 3$.
Then we have $S_{5,-2}|_F\geq C+C'$ for a moving curve $C'$ on $F$.   When $|C|$ and $|C'|$ are not composed of the same pencil, Theorem \ref{k3} (i.1) implies $\xi\geq \frac{4}{5}$ ($n=4$, $m_1=5$, $j_1=2$,  $j_2=1$, $\delta_2=2$, $\beta=\frac{1}{2}$).
Then $\varphi_{6,X}$ is birational by Theorem \ref{k3} (i.2), a contradiction. Otherwise, we have $S_{5,-2}|_F\ge 2C$.
Thus $\beta\ge\frac{4}{7}$ by  \cite[Proposition 3.5]{CHP}. Since $\alpha(7)\ge\frac{13}{6}>2$, we have $\xi\ge\frac{5}{7}$. By Theorem \ref{kf} (i.1), we have $\xi\ge\frac{4}{5}$ ($n=4$, $m_1=5$, $j_1=2$, $j_2=2$, $\beta=\frac{4}{7}$, $\xi\ge\frac{5}{7}$).  When $\xi>\frac{4}{5}$,  by Theorem \ref{kf} (i.2), $\varphi_{6,X}$ is birational ($m_1=5$, $j_1=2$, $j_2=2$), which is a contradiction. So the only possibility is $\xi=\frac{4}{5}$.

For (ii),  the condition $P_5(X)\ge 22$ and Lemma \ref{lem:Pm1 bir} (ii) imply that  $u_{5,-4}=2$.  If $|S_{5,-4}|_F|$ and $|C|$ are not composed of the same pencil, Theorem \ref{k2} (i) implies that $\varphi_{6,X}$ is birational ($m_1=5$, $j=4$, $\tilde{\delta}=2$, $\xi\geq \frac{4}{5}$, $\beta=\frac{3}{5}$), a contradiction.  If $|S_{5,-4}|_F|$ and $|C|$ are composed of the same pencil. When $\xi>\frac{4}{5}$, Theorem \ref{kf} (i.2) implies that $\varphi_{6,X}$ is birational ($m_1=5$, $j_1=4$, $j_2=1$, $\beta=\frac{3}{5}$). Thus the only possibility is $\xi=\frac{4}{5}$.

For (iii), we have $\mu\ge\frac{6}{5}$ and $\beta\ge\frac{6}{11}$ by our assumption. Since $\alpha(7)>2$, we have $\xi\ge\frac{5}{7}$. If $(S_{5,-1}\cdot C)\ge 4$, $\varphi_{6,X}$ is birational by Theorem \ref{k2} (ii) ($m_1=5$, $\tilde{\delta}=4$, $\mu\ge\frac{6}{5}$, $\beta\ge\frac{6}{11}$). So we may assume that $(S_{5,-1}|_F\cdot C)\le 3$. Thus we have $S_{5,-1}|_F\ge 2C+C_{-2}$, where $C_{-2}$ is a moving curve on $F$. If $|C_{-2}|$ and $|C|$ are not composed of the same pencil, $\varphi_{6,X}$ is birational by Theorem \ref{k3} (ii.2) ($m_1=5$, $j_1=1$, $j_2=2$, $\delta_2=2$, $\mu=\frac{6}{5}$). If $|C_{-2}|$ and $|C|$ are composed of the same pencil, we have $\beta\ge\frac{2}{3}$ by  \cite[Proposition 3.5]{CHP}. Since $\alpha(5)>1$, we have $\xi\ge\frac{4}{5}$. Since $\alpha(6)>2$, $\varphi_{6,X}$ is birational, which is a contradiction. Thus (iii) does not occur.

We are left to treat (iv). Since $h^0(M_5-5F)\ge 3$, Lemma \ref{lem:Pm1 bir} (ii) implies that $\varphi_{6,X}$ is birational, which is a contradiction.

Therefore we have $\xi=\frac{4}{5}$.\end{proof}

\begin{prop}\label{prop:P6}
Under the same assumption as that of Lemma \ref{lem:beta}, then
\begin{enumerate}
 \item when $P_6(X)\ge 35$, $\varphi_{6,X}$ is birational;
 \item when $32\le P_6(X)\le 34$, $\varphi_{6,X}$ is non-birational if and only if $\xi=\frac{4}{5}$.
 \end{enumerate}
\end{prop}
\begin{proof}
Set $m_1=6$. By Lemma \ref{lem:Pm1 bir} (i), we may and do assume that $u_{6,0}\le P_6(F)-1=17$.

{\bf Reduction to}: $\mathrm{dim}\psi_{6,0}(U_{6,0})\le 4$, $v_{6,-1}\le 3$, $v_{6,-2}\le 2$ {\bf and} $v_{6,-3}\le 1$.

By \cite[Proposition 3.4]{CHP}, we have $v_{6,0}\le 5$. If $\mathrm{dim}\psi_{6,0}(U_{6,0})=5$,
the Riemann-Roch formula implies that $\mathrm{deg} (M_6|_C)\geq 6$. Noting that $\mathrm{deg} (M_6|_C)\leq 6$,  $|M_6||_C$ must be complete. So we can write $|M_6||_C=|K_C+D|$ where $\mathrm{deg} (D)=4$. Thus $\varphi_{6,X}$ is birational. Hence we may assume that $\mathrm{dim}\psi_{6,0}(U_{6,0})\le 4$.

Suppose $v_{6,-1}\ge 4$. Then $M_6|_F\ge C+C_{-1}$ for some moving curve $C_{-1}$ on $F$ satisfying $h^0(C, C_{-1}|_C)\ge 4$. In particular, we have $(C_{-1}\cdot C)\ge 5$. By \cite[Proposition 3.7]{CHP} (ii), $\varphi_{6,X}$ is birational ($\mu=1$, $m_1=6$, $\delta_1=5$, $\beta=\frac{1}{2}$, $j=1$). We may assume that $v_{6,-1}\le 3$.

Suppose $v_{6,-2}\ge 3$. We have $M_6|_F\ge 2C+C_{-2}$, where $C_{-2}$ is a moving curve satisfying $h^0(C, C_{-2}|_C)\ge 3$. By Riemann-Roch formula, one has $(C_{-2}\cdot C)\ge 4$. We also have $$(C_{-2}\cdot C)\le (C_{-2}\cdot \sigma^*(K_{F_0}))\le ((M_6|_F-2C)\cdot\sigma^*(K_{F_0}))\le 4.$$ So $(C_{-2}\cdot C)=4$. By \cite[Proposition 3.7]{CHP} (i), $\varphi_{6,X}$ is birational ($\xi=\frac{2}{3}$, $\delta_1=4$, $\mu=1$, $j=2$, $m_1=6$). Thus we may assume that $v_{6,-2}\le 2$.

Now assume that $v_{6,-3}\ge 2$. Then $M_6|_F\ge 3C+C_{-3}$ for some moving curve $C_{-3}$ on $F$. In particular, we have $(C_{-3}\cdot C)\ge 2$. By \cite[Proposition 3.7]{CHP} (iii), we have $\xi\ge\frac{4}{5}$ ($n=4$, $\mu=1$, $m_1=6$, $j=3$, $\xi=\frac{2}{3}$, $\delta_1=2$). Thus $\varphi_{6,X}$ is birational by \cite[Proposition 3.7]{CHP} (i) ($\xi=\frac{4}{5}$, $\delta_1=2$, $j=3$, $\mu=1$, $m_1=6$). So we may assume that $v_{6,-3}\le 1$.
\medskip

{\bf Case 1.} $u_{6,0}\ge 11$.

If $\mathrm{dim} \psi_{6,0}(U_{6,0})=4$, one has $(M_6|_F\cdot C)\ge 5$ by Riemann-Roch formula. Hence
$\xi\ge\frac{5}{6}$. By our assumption, we have $M_6|_F\ge 4C$. Thus we get $\beta\ge\frac{2}{3}$. Since $$\alpha(6)\ge (6-1-1-\frac{1}{\beta})\cdot\xi\ge\frac{25}{12}>2,$$ $\varphi_{6,X}$ is birational by Theorem \ref{key-birat}.

If $\mathrm{dim} \psi_{6,0}(U_{6,0})\le 3$, we get $M_6|_F\ge 5C$ by our assumption. In particular, we have  $\beta\ge\frac{5}{6}$. By Lemma \ref{lem:beta2}, $\varphi_{6,X}$ is birational.
\medskip

{\bf Case 2.} $u_{6,0}=10$ and $P_6(X)\ge 31$.

If $\mathrm{dim} \psi_{6,0}(U_{6,0})\leq 3$ and $v_{6,-1}=3$, we have $h^0(F,M_6|_F-C)\ge 7$ and $M_6|_F\ge C+C_{-1}$, where $(C_{-1}\cdot C)\ge 4$.  On the other hand, by our assumption ($u_{6,0}=10$, $v_{6,-2}\le 2$ and $v_{6,-3}\le 1$), we have $M_6|_F\ge 4C$. In particular, we have $\beta\ge\frac{2}{3}$. By \cite[Proposition 3.7]{CHP} (ii), $\varphi_{6,X}$ is birational ($\mu=1$, $m_1=6$, $\delta_1=4$, $j=1$, $\beta=\frac{2}{3}$).

If  $\mathrm{dim} \psi_{6,0}(U_{6,0})\le 3$ and $v_{6,-1}\le 2$, we have $h^0(F,M_6|_F-C)\ge 8$. By our assumption ($u_{6,0}=10$, $v_{6,-2}\le 2$ and $v_{6,-3}\le 1$), we have $M_6|_F\ge 5C$. In particular, we get $\beta\ge\frac{5}{6}$. Lemma \ref{lem:beta2} implies that $\varphi_{6,X}$ is birational.

So we may and do assume that $\mathrm{dim}\psi_{6,0}(U_{6,0})=4$ throughout this {\bf Case}.  By Riemann-Roch formula, one has $\deg({M_6}|_C)\geq 5$. When $\deg({M_6}|_C)=5$, then $|M_6||_C$ must be complete and clearly $\varphi_{6,X}$ is birational.
Thus we can assume, from now on within this case, that  $(M_6|_F\cdot C)=6$.  In particular, $\xi=1$.
\medskip

Subcase 2.1. $u_{6,-1}\ge 7$.

We first consider the case when $\mathrm{dim}\psi_{6,0}(U_{6,-1})=4$. By our assumption, we have $(S_{6,-1}|_F\cdot C)=(M_6|_F\cdot C)=6$.  \cite[Proposition 3.6]{CHP} (1.2) implies that $\varphi_{6,X}$ is birational ($\beta=\frac{1}{2}$, $m_1=6$, $\delta=6$, $\mu=1$).

So we may assume that $\mathrm{dim}\psi_{6,0}(U_{6,-1})\le 3$. Thus we have $S_{6,-1}|_F\ge C+C_{-1}$, where $C_{-1}$ is a moving curve on $F$ satisfying $h^0(F, C_{-1})\ge 4$. If $\mathrm{dim} \psi_{6,-1}(H^0(F, C_{-1}))\ge 3$, we have $(C_{-1}\cdot C)\ge 4$ by Riemann-Roch formula. By Theorem \ref{k3} (i.2), $\varphi_{6,X}$ is birational ($\xi=1$, $j_1=j_2=1$, $\delta_2=4$, $\mu=1$, $\beta=\frac{1}{2}$).  We are left to treat the case when $\mathrm{dim} \psi_{6,-1}(H^0(F, C_{-1}))\le 2$. We have $S_{6,-1}|_F\ge 2C+C_{-2}$ where $C_{-2}$ is a moving curve on $F$. If $|C_{-2}|$ and $|C|$ are composed of the same pencil, we get $\beta\ge\frac{4}{7}$ by  \cite[Proposition 3.5]{CHP}. Since $\xi=1$, we have $(\pi^*(K_X)|_F)^2\ge\frac{4}{7}$. Lemma \ref{lem:restric vol} implies that $\varphi_{6, X}$ is birational. If $|C_{-2}|$ and $|C|$ are not composed of the same pencil, $\varphi_{6,X}$ is birational by Theorem \ref{k3} (ii.2) ($m_1=6$, $j_1=1$, $j_2=2$, $\delta_2=2$, $\xi=1$, $\mu=1$) .
\medskip

Subcase 2.2. $u_{6,-1}\le 6$,  $u_{6,-3}\ge 4$.

If $\psi_{6,-3}(U_{6,-3})\ge 3$, we have $(S_{6,-3}|_F\cdot C)\ge 4$. Therefore $(\pi^*(K_X)|_F)^2\ge\frac{5}{9}$ by Theorem \ref{k2} (iii) ($j=3$, $m_1=6$, $\tilde{\delta}=4$).  Then $\varphi_{6,X}$ is birational by Lemma \ref{lem:restric vol}.

If $\psi_{6,-3}(U_{6,-3})\le 2$, we have $S_{6,-3}|_F\ge C+C'$ where $C'$ is a moving curve. Thus we still have $(\pi^*(K_X)|_F\cdot S_{6,-3}|_F)\ge 2$. By Theorem \ref{k2} (iii), we have $(\pi^*(K_X)|_F)^2\ge\frac{5}{9}$.  Hence $\varphi_{6,X}$ is birational by Lemma \ref{lem:restric vol}.
\medskip

Subcase 2.3. $u_{6,-1}\le 6$,  $u_{6,-3}\le 3$,  $u_{6,-5}\ge 2$.

If $|S_{6,-5}|_F|$ and $|C|$ are composed of the same pencil, we have $\beta\ge\frac{6}{11}$ by  \cite[Proposition 3.5]{CHP}. Thus $(\pi^*(K_X)|_F)^2\ge \frac{6}{11}$. Lemma \ref{lem:restric vol} implies that $\varphi_{6,X}$ is birational.

If $|S_{6,-5}|_F|$ and $|C|$ are not composed of the same pencil, we have $((\pi^*(K_X))|_F\cdot S_{6,-5}|_F)\ge 1$. By Theorem \ref{k2} (iii), we have $(\pi^*(K_X)|_F)^2\ge\frac{6}{11}$ ($j=5$). Lemma \ref{lem:restric vol} implies that $\varphi_{6,X}$ is birational.
\medskip

Subcase 2.4. $u_{6,-1}\le 6$,  $u_{6,-3}\le 3$, $u_{6,-5}=1$ and $P_6(X)\ge 31$.

We have $h^0(M_6-5F)\ge 3$. Since $u_{6,-5}=1$, we have $\beta\ge\frac{7}{13}$ by \eqref{cri}. So $(\pi^*(K_X)|_F)^2\ge\frac{7}{13}$ and $\varphi_{6,X}$ is birational by Lemma \ref{lem:restric vol}.
\medskip

{\bf Case 3.} $u_{6,0}\le 9$, $u_{6,-1}\ge 8$.

If $(S_{6,-1}|_F\cdot C)=6$, we have $\xi=1$. By \cite[Proposition 3.6]{CHP}(1.2) ($m_1=6$, $\beta=\frac{1}{2}$, $\mu=1$, $\delta=6$),  $\varphi_{6,X}$ is birational.

If $(S_{6,-1}|_F\cdot C)\le 5$ and $\mathrm{dim} \psi_{6,0}(U_{6,-1})\ge 4$, the Riemann-Roch formula on $C$ tells that $$|S_{6,-1}||_C=|S_{6,-1}|_C|=|K_C+D|,$$ where $\mathrm{deg}(D)=3$. Thus $\varphi_{6,X}$ is birational.

If $\mathrm{dim}(\psi_{6,0}(U_{6,-1}))\le 3$, we have $S_{6,-1}|_F\ge C+C_{-1}$ where $C_{-1}$ is a moving curve satisfying $h^0(F, C_{-1})\ge 5$. By our reduction, we have $\mathrm{dim} \psi_{6,-1}(H^0(F, C_{-1}))\le 3$.  If $\mathrm{dim} \psi_{6,-1}(H^0(F, C_{-1}))=3$, we have $(C_{-1}\cdot C)\ge 4$ by Riemann-Roch formula. By Theorem \ref{k3} (i.2), $\varphi_{6,X}$ is birational ($j_1=1$, $j_2=1$, $m_1=6$, $\mu=1$, $\beta=\frac{1}{2}$, $\delta_2=4$).  If $\mathrm{dim} \psi_{6,-1}(H^0(F, C_{-1}))\le 2$, we have $S_{6,-1}|_F\ge 2C+C_{-2}$, where $C_{-2}$ is a moving curve on $F$ satisfying $h^0(F, C_{-2})\ge 3$. If $|C_{-2}|$ and $|C|$ are not composed of the same pencil, we have $(C_{-2}\cdot C)\ge 2$. By Theorem \ref{k3} (ii.1), we have $\xi\ge\frac{4}{5}$ ($n=4$, $m_1=6$, $j_1=1$, $j_2=2$, $\delta_2=2$, $\xi=\frac{2}{3}$, $\mu=1$). So $\varphi_{6, X}$ is birational by Theorem \ref{k3} (ii.2) ($j_1=1$, $j_2=2$, $\delta_2=2$, $\mu=1$, $\xi=\frac{4}{5}$). If $|C_{-2}|$ and $|C|$ are composed of the same pencil, we have $S_{6,-1}|_F\ge 4C$. By  \cite[Proposition 3.5]{CHP}, we have $\beta\ge\frac{5}{7}$. Lemma \ref{lem:beta2} implies that $\varphi_{6,X}$ is birational.
\medskip

{\bf Case 4.} $u_{6,0}\le 9$, $u_{6,-1}\le 7$ and $u_{6,-2}\ge 7$.

Note that $S_{6,-2}|_F\ge M_4|_F\ge 2\sigma^*(K_{F_0})$. So $S_{6,-2}|_F$ is a big divisor.

If $(S_{6,-2}|_F\cdot C)\ge 5$, $\varphi_{6,X}$ is birational by Theorem \ref{k2} (ii)$'$ ($m_1=6$, $j=2$, $\beta=\frac{1}{2}$).

 If $(S_{6,-2}|_F\cdot C)\le 4$, we have $S_{6,-2}|_F\ge C+C_{-1}$, where $C_{-1}$ is a moving curve satisfying $h^0(F, C_{-1})\ge 4$. If $h^0(C, C_{-1}|_C)\ge 3$, we have $(C_{-1}\cdot C)=4$ by Riemann-Roch formula and our assumption $(S_{6,-2}|_F\cdot C)\le 4$. By Theorem \ref{k3} (i.2), $\varphi_{6,X}$ is birational ($j_1=2$, $j_2=1$, $\delta_2=4$, $m_1=6$, $\xi=\frac{2}{3}$, $\beta=\frac{1}{2}$, $\mu=1$ ) . We may assume that $h^0(C, C_{-1}|_C)\le 2$. Thus we have $S_{6,-2}|_F\ge 2C+C_{-2}$, where $C_{-2}$ is a moving curve on $F$. If $|C_{-2}|$ and $|C|$ are not composed of the same pencil, we have $(C_{-2}\cdot C)\ge 2$. By Theorem \ref{k3} (i.2), $\varphi_{6,X}$ is birational ($j_1=2$, $j_2=2$, $\xi=\frac{2}{3}$, $\delta_2=2$, $m_1=6$, $\beta=\frac{1}{2}$). If $|C_{-2}|$ and $|C|$ are composed of the same pencil, we have $S_{6,-2}|_F\ge 3C$. By Theorem \ref{kf} (ii.1), we have $\xi\ge\frac{4}{5}$ ($n=4$, $m_1=6$, $j_1=2$, $j_2=3$, $\mu=1$, $\xi=\frac{2}{3}$). Thus $\varphi_{6,X}$ is birational by Theorem \ref{kf} (ii.2) ($j_1=2$, $j_2=3$, $m_1=6$, $\mu=1$, $\xi=\frac{4}{5}$).
\medskip

{\bf Case 5.} $u_{6,0}\le 9$, $u_{6,-1}\le 7$, $u_{6,-2}\le 6$ and $u_{6,-3}\ge 5$.

If $(S_{6,-3}|_F\cdot C)\ge 4$, we have $\xi\ge\frac{7}{9}$ by Theorem \ref{k2} (iii) ($j=3$, $m_1=6$, $\tilde{\delta}=4$). By Theorem \ref{k2} (i) and (ii),  $\varphi_{6,X}$ is birational ($j=3$, $\tilde{\delta}\ge 4$, $\xi=\frac{7}{9}$, $m_1=6$, $\beta=\frac{1}{2}$).

If $(S_{6,-3}|_F\cdot C)\le 3$, we have $S_{6,-3}|_F\ge C+C_{-1}$, where $C_{-1}$ is a moving curve satisfying $h^0(F, C_{-1})\ge 3$. If $|C_{-1}|$ and $|C|$ are composed of the same pencil, we have $S_{6,-3}|_F\ge 3C$. By  \cite[Proposition 3.5]{CHP}, we have $\beta\ge\frac{2}{3}$. By Theorem \ref{kf} (i.1), we have $\xi\ge\frac{4}{5}$ ($n=4$, $m_1=6$, $j_1=j_2=3$, $\xi=\frac{2}{3}$, $\beta=\frac{2}{3}$).  Then $\varphi_{6,X}$ is birational by Theorem \ref{kf} (i.2) ($j_1=3$, $j_2=3$, $\xi=\frac{4}{5}$, $m_1=6$, $\beta=\frac{2}{3}$). If $|C_{-1}|$ and $|C|$ are not composed of the same pencil, we have $(C_{-1}\cdot C)\ge 2$. By Theorem \ref{k3} (i.2), $\varphi_{6,X}$ is birational ($j_1=3$, $j_2=2$, $\xi=\frac{2}{3}$, $m_1=6$, $\delta_2=2$, $\beta=\frac{1}{2}$).
\medskip

{\bf Case 6.} $u_{6,0}\le 9$, $u_{6,-1}\le 7$, $u_{6,-2}\le 6$, $u_{6,-3}\le 4$ and $u_{6,-4}\ge 4$.

If $(S_{6,-4}|_F\cdot C)\ge 4$, we get $\xi\ge\frac{4}{5}$ by Theorem \ref{k2} (iii) ($j=4$, $\tilde{\delta}=4$, $m_1=6$). Hence $\varphi_{6,X}$ is birational by Theorem \ref{k2} (i) ($\tilde{\delta}=4$, $j=4$, $\xi=\frac{4}{5}$, $\beta=\frac{1}{2}$).

If $(S_{6,-4}|_F\cdot C)\le 3$, we get $S_{6,-4}|_F\ge C+C_{-1}$, where $C_{-1}$ is a moving curve. If $|C_{-1}|$ and $|C|$ are composed of the same pencil, we have $S_{6,-4}|_F\ge 2C$. By  \cite[Proposition 3.5]{CHP}, we have $\beta\ge\frac{3}{5}$. By Theorem \ref{kf} (i.1), we have $\xi\ge\frac{4}{5}$ ($n=4$, $m_1=6$, $j_1=4$, $j_2=2$, $\xi=\frac{2}{3}$, $\beta=\frac{3}{5}$). Then $\varphi_{6,X}$ is birational by Theorem \ref{kf} (i.2) ($j_1=4$, $j_2=2$, $\xi=\frac{4}{5}$, $\beta=\frac{3}{5}$). If $|C_{-1}|$ and $|C|$ are not composed of the same pencil, we have $(C_{-1}\cdot C)\ge 2$. By Theorem \ref{k3} (i.1), we have $\xi\ge\frac{4}{5}$ ($n=4$, $m_1=6$, $j_1=4$, $j_2=1$, $\beta=\frac{1}{2}$, $\xi=\frac{2}{3}$, $\delta_2=2$).
Theorem \ref{k3} (i.2) implies that  $\varphi_{6,X}$ is birational ($j_1=4$, $j_2=1$, $m_1=6$, $\xi=\frac{4}{5}$, $\beta=\frac{1}{2}$).
\medskip

{\bf Case 7.} $u_{6,0}\le 9$, $u_{6,-1}\le 7$, $u_{6,-2}\le 6$, $u_{6,-3}\le 4$, $u_{6,-4}\le 3$ and $u_{6,-5}\ge 3$.

If $(S_{6,-5}|_F\cdot C)\ge 4$. By the same argument as in {\bf Case 6}, $\varphi_{6,X}$ is birational (Note that we have $S_{6,-4}|_F\ge S_{6,-5}|_F$).

If $2\le (S_{6,-5}|_F\cdot C)\le 3$, we have $S_{6,-5}|_F\ge C$ by Riemann-Roch formula. By  \cite[Proposition 3.5]{CHP}, we have $\beta\ge\frac{6}{11}$. By Theorem \ref{k2} (iv), we have $\xi\ge\frac{4}{5}$ ($n=4$, $\xi=\frac{2}{3}$, $m_1=6$, $j=5$, $\beta=\frac{6}{11}$, $\tilde{\delta}=2$). Then  $\varphi_{6,X}$ is birational by Theorem \ref{k2}(i) ($j=5$, $m_1=6$, $\tilde{\delta}=2$, $\xi=\frac{4}{5}$, $\beta=\frac{6}{11}$ ).

If $|S_{6,-5}|_F|$ and $|C|$ are composed of the same pencil, we get $S_{6,-5}|_F\ge 2C$.  By  \cite[Proposition 3.5]{CHP}, we have $\beta\ge\frac{7}{11}$. By Theorem \ref{kf} (i.1), we have $\xi\ge\frac{4}{5}$ ($n=4$, $m_1=6$, $j_1=5$, $j_2=2$, $\beta=\frac{7}{11}$, $\xi=\frac{2}{3}$). Then  $\varphi_{6,X}$ is birational by Theorem \ref{kf} (i.2) ($j_1=5$, $j_2=2$, $m_1=6$, $\xi=\frac{4}{5}$, $\beta=\frac{7}{11}$).
\medskip

{\bf Case 8.} $u_{6,0}\le 9$, $u_{6,-1}\le 7$, $u_{6,-2}\le 6$, $u_{6,-3}\le 4$, $u_{6,-4}\le 3$, $u_{6,-5}\le2$, $u_{6,-6}=2$.

If $|S_{6,-6}|_F|$ and $|C|$ are composed of the same pencil, we get $\beta\ge\frac{7}{12}$ by  \cite[Proposition 3.5]{CHP}. By Theorem \ref{kf} (i.1), we have $\xi\ge\frac{4}{5}$ ($n=4$, $m_1=6$, $j_1=6$, $j_2=1$, $\beta=\frac{7}{12}$, $\xi=\frac{2}{3}$).  Then $\varphi_{6,X}$ is birational by Theorem \ref{kf} (i.2) ($j_1=6$, $j_2=1$, $m_1=6$, $\xi=\frac{4}{5}$, $\beta=\frac{7}{12}$).

If $|S_{6,-6}|_F|$ and $|C|$ are not composed of the same pencil, we have $(S_{6,-6}|_F\cdot C)\ge 2$. By Theorem \ref{k2} (iv), we have $\xi\ge\frac{5}{7}$ ($n=6$, $m_1=6$, $j=6$, $\beta=\frac{1}{2}$, $\xi=\frac{2}{3}$). One has $\xi\ge\frac{4}{5}$ by Theorem \ref{k2} (iv) ($n=4$, $m_1=6$, $j=6$, $\beta=\frac{1}{2}$, $\xi=\frac{5}{7}$). Thus $\varphi_{6,X}$ is birational by Theorem \ref{k2} (i)$'$ ($\tilde{\delta}=2$, $j=6$, $\xi=\frac{4}{5}$, $m_1=6$, $\beta=\frac{1}{2}$) .
\medskip

{\bf Case 9.} $u_{6,0}\le 9$, $u_{6,-1}\le 7$, $u_{6,-2}\le 6$, $u_{6,-3}\le 4$, $u_{6,-4}\le 3$, $u_{6,-5}\le 2$, $u_{6,-6}=1$ and $P_6(X)\ge 35$.

We have $\mu\ge\frac{3}{2}>\frac{4}{3}$ by our assumption. So  $\varphi_{6,X}$ is birational by Lemma \ref{lem:zeta}.

Now suppose $32\le P_6(X)\le 34$. Lemma \ref{lem:charac K1} implies that we only need to consider the direction by assuming the non-birationality of $\varphi_{6,X}$.
By the arguments  in {\bf Case 1}-{\bf Case 9}, it suffices to consider one of the following cases: Case i $\sim$ Case iii.
\medskip

{\bf Case i.} $u_{6,0}=9$, $\mathrm{dim} \psi_{6,0}(U_{6,0})\le 4$.

 ($\dag$) We first treat the case when $\mathrm{dim} \psi_{6,0}(U_{6,0})=4$. We have $(M_6|_F\cdot C)\ge 5$ by Riemann-Roch formula.  If $(M_6|_F\cdot C)=5$, Rimann-Roch formula implies that $|M_6||_C$ is a complete linear system $|K_C+D|$, where $\mathrm{deg} D=3$. So $\varphi_{6,X}$ is birational, which is a contradiction. So we have $(M_6|_F\cdot C)=6$, which implies $\xi=1$. We will prove that this can not happen at all. By Lemma \ref{lem:restric vol}, we have $\beta=\frac{1}{2}$.
\medskip

Subcase i.a. $u_{6,-1}=7$, $\mathrm{dim} \psi_{6,0}(U_{6,0})=4$.

If $\dim \psi_{6,0}(U_{6,-1})\geq 4$,  we have $(S_{6,-1}|_F\cdot C)\geq 5$ by  Riemann-Roch formula. For the case $(S_{6,-1}|_F\cdot C)\geq 6$, $\varphi_{6,X}$ is birational by \cite[Proposition 3.6]{CHP} (1.2) ($m_1=6$, $\delta=6$, $\beta=\frac{1}{2}$, $\mu=1$).
For the case $(S_{6,-1}|_F\cdot C)=5$,  the linear system $|S_{6,-1}||_C$ must be the complete one, i.e. $|S_{6,-1}|_C|$, due to Riemann-Roch formula as well.  In fact, $|S_{6,-1}|_C|=|K_C+D|$ with $\deg(D)=3$. Clearly, $\varphi_{6,X}$ is birational.

If $\dim \psi_{6,0}(U_{6,-1})\leq 3$,
%If $(S_{6,-1}|_F\cdot C)\le 3$, {\Cred{we have $\dim\psi_{6,0} H^0(S_{6,-1}|_F)\leq 2$ and so}}
%$S_{6,-1}|_F\ge 3C$, which implies that $\beta\ge\frac{4}{7}>\frac{1}{2}$ by  \cite[Proposition 3.5]{CHP}, which is a contradiction. Thus we have $(S_{6,-1}|_F\cdot C)=4$ by our assumption. Therefore
we have $S_{6,-1}|_F\ge C+C_{-1}^{'}$, where $C_{-1}^{'}$ is a moving curve on $F$ satisfying $h^0(F, C_{-1}^{'})\geq 4$.
When $(C_{-1}^{'}\cdot C)\geq 4$, then $\varphi_{6,X}$ is birational by Theorem \ref{k3} (i.2) ($m_1=6$, $j_1=j_2=1$, $\delta_2=4$, $\beta=\frac{1}{2}$, $\mu=1$), which is a contradiction. So we have $(C_{-1}^{'}\cdot C)\le 3$, which implies that  $S_{6,-1}|_F\ge 2C+C_{-2}^{'}$, where $C_{-2}^{'}$ is a moving curve on $F$. When $|C_{-2}^{'}|$ and $|C|$ are composed of the same pencil, we have $\beta\ge\frac{4}{7}$ by  \cite[Proposition 3.5]{CHP}.
Then we have $(\pi^*(K_X)|_F)^2\geq \frac{4}{7}>\frac{1}{2}$, which means $\varphi_{6,X}$ is birational by Lemma \ref{lem:restric vol} (a contradiction).

 When $|C_{-2}^{'}|$ and $|C|$ are not composed of the same pencil, we have $(C_{-2}^{'}\cdot C)\ge 2$. Theorem \ref{k3} (ii.2) implies that $\varphi_{6,X}$ is birational ($m_1=6$, $j_1=1$, $j_2=2$, $\delta_2=2$, $\xi=1$, $\mu=1$), which contradicts to our assumption.

In a word, Subcase i.a does not occur.
\medskip

Subcase i.b. $u_{6,-1}\le 6$, $u_{6,-3}\le 4$, $u_{6,-4}\le 3$, $u_{6,-5}\le 2$, $u_{6,-6}=1$ and $\mathrm{dim} \psi_{6,0}(U_{6,0})=4$.

Since $P_6(X)\ge 32$, we have $M_6\ge 7F$ by our assumption. By  Inequality \eqref{cri}, we have $\beta\ge\frac{7}{13}$ and $(\pi^*(K_X)|_F)^2>\frac{1}{2}$, which is a contradiction by Lemma \ref{lem:restric vol}. Hence this subcase does not occur either.
\medskip

($\ddag$) We then treat the case when $\mathrm{dim} \psi_{6,0}(U_{6,0})\le 3$. Since $$\alpha(6)\ge (6-1-1-\frac{1}{\beta})\cdot \xi>1,$$
$\varphi_{6,X}$ is generically finite, which
implies that $\mathrm{dim} \psi_{6,0}(U_{6,0})\ge 3$. We may and do assume that $\mathrm{dim} \psi_{6,0}(U_{6,0})=3$ throughout the rest of this case. We have $M_6|_F\ge C+C_{-1}$, where $C_{-1}$ is a moving curve on $F$ satisfying $h^0(F,C_{-1})\ge 6$.

If $(C_{-1}\cdot C)\le 3$, we have $C_{-1}\ge 2C+C_{-2}$, where $C_{-2}$ is a moving curve on $F$. If $|C_{-2}|$ and $|C|$ are not composed of the same pencil, we have $(C_{-2}\cdot C)\ge 2$. By \cite[Proposition 3.7]{CHP} (iii), we have $\xi\ge\frac{4}{5}$ ($m_1=6$, $j=3$, $\delta_1=2$, $\mu=1$, $\xi\ge\frac{2}{3}$, $n=4$). \cite[Proposition 3.7]{CHP} (i) implies that $\varphi_{6,X}$ is birational ($m_1=6$, $j=3$, $\delta_1=2$, $\mu=1$, $\xi\ge\frac{4}{5}$), which contradicts to our assumption. Thus $|C_{-2}|$ and $|C|$ are composed of the same pencil. We get $M_6|_F\ge 4C$. In particular, we have $\beta\ge\frac{2}{3}$. Since $\alpha(7)>2$, we have $\xi\ge\frac{5}{7}$. We have $\alpha(5)\ge \frac{15}{14}>1$. So $\xi\ge\frac{4}{5}$. We can get $\alpha(6)>2$ when $\xi>\frac{4}{5}$. Thus we have $\xi=\frac{4}{5}$.

If $(C_{-1}\cdot C)\geq 4$, \cite[Proposition 3.7]{CHP} (ii) implies that $\varphi_{6,X}$ is birational whenever $\beta>\frac{1}{2}$.  Thus we need to study the situation with  $\beta=\frac{1}{2}$. Taking $n=9$, $10$, $11$, $12$, $13$, respectively, and run \cite[Proposition 3.7]{CHP} (iii),  one finally gets $\xi\geq \frac{6}{7}$.  So we will work under the constraints:
$\xi\ge\frac{6}{7}$ and $\beta=\frac{1}{2}$, throughout the rest of this case.

%we have $(n+1)\xi\ge \ulcorner(n-7)\xi \urcorner+6$ for any positive integer $n>7$ ($m_1=6$, $j=1$, $\delta_1=4$). Take $n=9$, we get $\xi\ge\frac{4}{5}$. Take $n=10$, we get $\xi\ge\frac{9}{11}$. Take $n=11$, we get $\xi\geq\frac{5}{6}$. Take $n=12$, we get $\xi\ge\frac{11}{13}$. Take $n=13$, we get $\xi\geq\frac{6}{7}$. Thus we may assume that $
\medskip

Subcase i.c. $u_{6,-1}=7$, $\mathrm{dim} \psi_{6,0}(U_{6,0}) \leq 3$.

%Since $M_6|_F\ge S_{6,-1}|_F$, we have $(S_{6,-1}|_F\cdot C)\le 4$. If $(S_{6,-1}|_F\cdot C)\le 3$, we have $S_{6,-1}|_F\ge 3C$, which implies that $\beta\ge\frac{4}{7}>\frac{1}{2}$ by  \cite[Proposition 3.5]{CHP}, which is a contradiction. Thus we have $(S_{6,-1}|_F\cdot C)=4$ by our assumption. Therefore we have $S_{6,-1}|_F\ge C+C_{-1}^{'}$, where $C_{-1}^{'}$ is a moving curve on $F$ satisfying $h^0(F, C_{-1}^{'})\ge 4$.

%If $(C_{-1}^{'}\cdot C)=4$, $\varphi_{6,X}$ is birational by Theorem \ref{k3} (i), which is a contradiction. So we have $(C_{-1}^{'}\cdot C)\le 3$, which implies that we have $S_{6,-1}|_F\ge 2C+C_{-2}^{'}$, where $C_{-2}^{'}$ is a moving curve on $F$. When $|C_{-2}^{'}|$ and $|C|$ are composed of the same pencil, we have $\beta\ge\frac{4}{7}$ by  \cite[Proposition 3.5]{CHP}, which contradicts to our assumption. When $|C_{-2}^{'}|$ and $|C|$ are not composed of the same pencil, we have $(C_{-2}^{'}\cdot C)\ge 2$. Theorem \ref{k3} (ii) implies that $\varphi_{6,X}$ is birational ($m_1=6$, $j_1=1$, $j_2=2$, $\delta_2=2$, $\xi=\frac{6}{7}$, $\mu=1$), which contradicts to our assumption.

Clearly, one has $\dim \psi_{6,0}(U_{6,-1}) \leq 3$ , which is parallel to the second part of Subcase i.a. We have $S_{6,-1}|_F\ge C+C_{-1}^{'}$, where $C_{-1}^{'}$ is a moving curve on $F$ satisfying $h^0(F, C_{-1}^{'})\geq 4$.
When $(C_{-1}^{'}\cdot C)\geq 4$, then $\varphi_{6,X}$ is birational by Theorem \ref{k3} (i.2) ($m_1=6$, $j_1=j_2=1$, $\delta_2=4$, $\beta=\frac{1}{2}$, $\mu=1$), which is a contradiction. So we have $(C_{-1}^{'}\cdot C)\le 3$, which implies that  $S_{6,-1}|_F\ge 2C+C_{-2}^{'}$, where $C_{-2}^{'}$ is a moving curve on $F$. When $|C_{-2}^{'}|$ and $|C|$ are composed of the same pencil, we have $\beta\ge\frac{4}{7}$ by  \cite[Proposition 3.5]{CHP}, a contradiction to our assumption $\beta=\frac{1}{2}$. When $|C_{-2}^{'}|$ and $|C|$ are not composed of the same pencil, we have $(C_{-2}^{'}\cdot C)\ge 2$. Theorem \ref{k3} (ii.2) implies that $\varphi_{6,X}$ is birational ($m_1=6$, $j_1=1$, $j_2=2$, $\delta_2=2$, $\xi=\frac{6}{7}$, $\mu=1$), which contradicts to our assumption.
\medskip

Subcase i.d. $u_{6,-1}\le 6$, $u_{6,-3}=4$.

If $(S_{6,-3}|_F\cdot C)\ge 4$,   Theorem \ref{k2} (i) implies that $\varphi_{6,X}$ is birational ($m_1=6$, $j=3$, $\beta=\frac{1}{2}$, $\tilde{\delta}=4$, $\xi=\frac{6}{7}$), which contradicts to our assumption.

If $(S_{6,-3}|_F\cdot C)\le 3$, we have $S_{6,-3}|_F\ge C+C'$, where $C'$ is a moving curve on $F$. When $|C'|$ and $|C|$ are composed of the same pencil, we have $\beta\ge\frac{5}{9}>\frac{1}{2}$ by  \cite[Proposition 3.5]{CHP}, which contradicts to assumption. Then $|C'|$ and $|C|$ are not composed of the same pencil, we get $(C'\cdot C)\ge 2$.  Then $\varphi_{6,X}$ is birational by Theorem \ref{k3} (i.2) ($m_1=6$, $j_1=3$, $j_2=1$, $\delta_2=2$, $\beta=\frac{1}{2}$, $\xi\ge\frac{6}{7}$), which contradicts to our assumption.
\medskip

Subcase i.e. $u_{6,-1}\le 6$, $u_{6,-3}\le 3$, $u_{6,-5}=2$.

The assumption $\beta=\frac{1}{2}$ implies that $|S_{6,-5}|_F|$ and $|C|$ are not composed of the same pencil. Then $(S_{6,-5}|_F\cdot C)\ge 2$. Theorem \ref{k2} (i)$'$ implies that $\varphi_{6,X}$ is birational ($m_1=6$, $j=5$, $\tilde{\delta}=2$, $\xi=\frac{6}{7}$), which contradicts to our assumption. Thus this subcase does not occur.
\medskip

Subcase i.f. $u_{6,-1}\le 6$, $u_{6,-3}\le 3$,  $u_{6,-5}=1$, $P_6(X)\ge 30$.

We have $h^0(M_6-5F)\ge 3$. Thus we have $\beta\ge\frac{7}{13}>\frac{1}{2}$ by Inequality \eqref{cri}, which contradicts to our assumption.
\medskip

{\bf Case ii.} $u_{6,0}\le 8$, $u_{6,-1}\le 7$, $u_{6,-2}=6$, $u_{6,-3}\le 4$, $u_{6,-4}\le 3$, $u_{6,-5}\le 2$, $u_{6,-6}=1$.

If $(S_{6,-2}|_F\cdot C)\le 3$, we have $S_{6,-2}|_F\ge 2C+C_{-2}$, where $C_{-2}$ is a moving curve on $F$. When $|C_{-2}|$ and $|C|$ are composed of the same pencil, we have $\xi\ge\frac{4}{5}$ by Theorem \ref{kf} (ii.1) ($n=4$, $m_1=6$, $j_1=2$, $j_2=3$, $\xi=\frac{2}{3}$, $\mu=1$). Thus $\varphi_{6,X}$ is birational by Theorem \ref{kf} (ii.2) ($m_1=6$, $j_1=2$, $j_2=3$, $\xi=\frac{4}{5}$, $\mu=1$), which is a contradiction. When $|C_{-2}|$ and $|C|$ are not composed of the same pencil, we have $(C_{-2}\cdot C)\ge 2$. By Theorem \ref{k3} (i.2), $\varphi_{6,X}$ is birational ($m_1=6$, $j_1=2$, $j_2=2$, $\xi=\frac{2}{3}$, $\delta_2=2$, $\beta=\frac{1}{2}$), which contradicts to our assumption.

Thus we have $(S_{6,-2}|_F\cdot C)\ge 4$. Theorem \ref{k2} (i) and (ii) imply that $\varphi_{6,X}$ is birational if $\beta>\frac{1}{2}$ ($m_1=6$, $j=2$, $\tilde{\delta}\ge 4$, $\xi=\frac{2}{3}$, $\mu=1$). So we have $\beta=\frac{1}{2}$ by our assumption. But our assumption in this case gives $h^0(M_6-6F)\ge 2$. Since $u_{6,-6}=1$, we have $\beta\ge\frac{7}{13}$, which is a contradiction.
\medskip

{\bf Case iii.} $u_{6,0}\le 8$, $u_{6,-1}\le 7$, $u_{6,-2}\le 5$, $u_{6,-3}\le 4$, $u_{6,-4}\le 3$, $u_{6,-5}\le 2$, $u_{6,-6}=1$.

We have $h^0(M_6-6F)\ge 3$ since $P_6(X)\ge 32$. Since $u_{6,-6}=1$, we have $\mu\ge\frac{4}{3}$. Then one gets $\beta\ge\frac{4}{7}$ by Inequality \eqref{cri}. We have $\alpha(7)\ge\frac{7}{3}>2$, so $\xi\ge\frac{5}{7}$. Since $\alpha(5)>1$, we have $\xi\ge\frac{4}{5}$. When $\xi>\frac{4}{5}$, we have $\alpha(6)>2$, which implies that $\varphi_{6,X}$ is birational. So the only possibility is $\xi=\frac{4}{5}$.
\end{proof}

\subsection{Estimation of the canonical volume}

We go on working under the same assumption as that of Lemma \ref{lem:beta}.

%The following lemma is an easy consequence of Hodge index theorem.

\begin{lem}\label{lem:estimation}
Let $\pi\colon X'\rightarrow X$ be any birational morphism where $X'$ is nonsingular and projective. Assume that $|M|$ is a base point free linear system on $X'$. Denote by $S$ a general member of $|M|$. Then the following inequality holds:
\begin{align}\label{eq:estimation1}
((\pi^*(K_X)|_S)^2)^2\ge K_X^3\cdot(\pi^*K_X|_S\cdot S|_S)
\end{align}
\end{lem}
\begin{proof}
Take a sufficiently large and divisible integer $m$ such that the linear system $|\pi^*(mK_X)|$ is base point free. Denote by $S_m$ a general member of $|\pi^*(mK_X)|$. By Bertini's theorem, $S_m$ is a smooth projective surface of general type. On the surface $S_m$, by Hodge index theorem, we have
\begin{align*}
(\pi^*(K_X)|_{S_m}\cdot S|_{S_m})^2\ge(\pi^*(K_X)|_{S_m})^2\cdot (S|_{S_m})^2,
\end{align*}
which implies \eqref{eq:estimation1}.
\end{proof} 

Let $m$ be a positive integer and $l$ be another integer satisfying $0\le l\le m$.
Assume that $h^0(M_m-lF)\geq 2$. Denote by $|S_{m,-l}|$ the moving part of $|M_m-lF|$. Modulo further blowups, we may assume that $|S_{m,-l}|$ is base point free. 
Multiplying the following inequality with $\pi^*(K_X)^2$:
$$M_m\geq jF+S_{m,-l}$$
while applying \eqref{eq:estimation1}, we have
\begin{align}\label{eq:estimation1.2}
K_{X}^3&\ge \frac{l(\pi^*(K_X)|_F)^2+\sqrt{K_X^3\cdot (\pi^*(K_X)|_{S_{m,-l}}\cdot {S_{m,-l}}|_{S_{m,-l}}) }}{m}\\ \nonumber
&\ge\frac{l(\pi^*(K_X)|_F)^2+\sqrt{(m-l)K_X^3\cdot(\pi^*(K_X)|_F\cdot {S_{m,-l}}|_F})}{m}.
\end{align}
For the last inequality, we note that $M_m\geq mF$, which implies $S_{m,-l}\geq (m-l)F$.

\begin{prop}\label{prop:esti of volume} Keep the same assumption as that of Lemma \ref{lem:beta}. Suppose that $\varphi_{6,X}$ is not birational, $\xi\neq\frac{2}{3}$ and $\xi\neq \frac{4}{5}$. Then the following holds:
\begin{enumerate}
\item $K_X^3\ge\frac{5}{14}$;
\item when $P_6(X)\ge 26$, $K_X^3\ge \frac{11}{28}$;
\item when $P_6(X)\ge 27$, $K_X^3>0.4328$;
\item when $P_6(X)\ge 28$, $K_X^3>0.4714$;
\item when $P_6(X)\ge 31$, $K_X^3\ge\frac{8}{15}$.
\end{enumerate}

\end{prop}
\begin{proof} Since $\alpha(7)\ge (7-1-1-\frac{1}{\beta})\cdot \xi>2$, we have $\xi\ge\frac{5}{7}$.

Statement (1) follows from inequality \eqref{kcube} with $\beta\ge\frac{1}{2}$ and $\xi\ge\frac{5}{7}$.

By the arguments in {\bf Case 1}$\sim${\bf Case 9} in the proof of Proposition \ref{prop:P6}, one of the following cases holds:
\medskip

{\bf Case I.} $u_{6,0}=10$, $\mathrm{dim}\psi_{6,0}(U_{6,0})=4$ and $P_6(X)\le 30$.  ($\Rightarrow K_X^3\ge \frac{1}{2}$)

We have $(M_6|_F\cdot C)\ge 5$ by Riemann-Roch formula. If $(M_6|_F\cdot C)=5$, $|M_6||_C$ is a complete linear system whose general member has degree $5$, which implies that $\varphi_{6,X}$ is birational, which contradicts to our assumption. Thus we have $(M_6|_F\cdot C)=6$. We get $\xi=1$. Therefore we have $K_X^3\ge\frac{1}{2}$.
\medskip

{\bf Case IIa.} $u_{6,0}=9$, $\mathrm{dim}\psi_{6,0}(U_{6,0})=4$. ($\Rightarrow K_X^3\ge \frac{1}{2}$)

As we  have seen, one has $(M_6|_F\cdot C)\geq 5$ and $''=5''$ implies the birationality of $\varphi_{6,X}$. Hence $(M_6|_F\cdot C)=6$. In particular, $\xi=1$.
By Lemma \ref{lem:restric vol} and the assumption, we have $\beta=\frac{1}{2}$ and $(\pi^*(K_X)|_F)^2=\frac{1}{2}$. Thus $K_X^3\ge\frac{1}{2}$.
\smallskip

\noindent{\bf Claim}. We have $u_{6,-1}\le 6$, $u_{6,-3}\le 3$, $u_{6,-5}=1$, $P_6(X)\le 29$.

In fact, $u_{6,-1}\le 6$ follows from the proof of Proposition \ref{prop:P6} (see Subcase i.a.).

From the proof of Proposition \ref{prop:P6}, we have $u_{6,-3}\le 4$ and $u_{6,-5}\le 2$.

Suppose  $u_{6,-3}=4$. If $(S_{6,-3}|_F\cdot C)\ge 4$, by Theorem \ref{k2} (iii) ($m_1=6$, $j=3$), we have $$\pi^*(K_X)|_F\ge \frac{1}{3}C+\frac{1}{9}S_{6,-3}|_F.$$ Since $\xi=1$, we have $(\pi^*(K_X)|_F)^2\ge\frac{5}{9}>\frac{1}{2}$, which contradicts to our assumption. So $(S_{6,-3}|_F\cdot C)\le 3$, which gives $S_{6,-3}|_F\ge C+C_{-1}$, where $C_{-1}$ is a moving curve on $F$. Using the same argument as above, we can get $(\pi^*(K_X)|_F)^2\ge\frac{5}{9}>\frac{1}{2}$, which is a contradiction. So $u_{6,-3}\le 3$.

By the similar argument as above, we also sees that $u_{6,-5}=1$.

If $P_6(X)\ge 30$, we have $h^0(M_6-5F)\ge 3$. By Inequality \eqref{cri}, we have $\beta\ge\frac{7}{13}>\frac{1}{2}$, which is a contradiction.
\medskip

{\bf Case IIb.} $u_{6,0}=9$, $\mathrm{dim}\psi_{6,0}(U_{6,0})\le 3$.  ($\Rightarrow K_X^3\ge \frac{10}{21}$)

One has $M_6|_F\ge C+C_{-1}$, where $C_{-1}$ is a moving curve on $F$ satisfying $h^0(F, C_{-1})\ge 6$. By the argument in Case i of Proposition \ref{prop:P6} and the assumption $\xi\neq \frac{4}{5}$, we know $(C_{-1}\cdot C)\ge 4$, $\xi\ge\frac{6}{7}$, $\beta=\frac{1}{2}$, $u_{6,-1}\le 6$, $u_{6,-3}\le 3$, $u_{6,-5}=1$ and $P_6(X)\le 29$.

We have $$(\pi^*(K_X)|_F)^2\ge\frac{(\pi^*(K_X)|_F\cdot C)+(\pi^*(K_X)|_F\cdot C_{-1})}{6}\ge \frac{10}{21}.$$ In particular, we have $K_X^3\ge\frac{10}{21}$.
\medskip

{\bf Case IIIa.} $u_{6,0}=8$, $\mathrm{dim}\psi_{6,0}(U_{6,0})=4$. ($\Rightarrow K_X^3\ge\frac{1}{2}$)

For the similar reason, we have $(M_6|_F\cdot C)=6$ and so $\xi=1$. By the same argument as in {\bf Case IIa}, we have $u_{6,-1}\le 6$, $u_{6,-3}\le 3$, $u_{6,-5}=1$, $P_6(X)\le 28$. In particular, we have $K_X^3\ge\frac{1}{2}$.
\medskip

{\bf Case IIIb.}  $u_{6,0}=8$, $\mathrm{dim}\psi_{6,0}(U_{6,0})\le 3$, $M_6|_F\ge C+C_{-1}$ with $(C_{-1}\cdot C)=4$.  ($\Rightarrow K_X^3\ge \frac{10}{21}$)

By the same argument as in {\bf Case i} of Proposition \ref{prop:P6} (Subcase i.c $\sim$ Subcase i.f), we  know $\xi\ge\frac{6}{7}$, $\beta=\frac{1}{2}$, $u_{6,-1}\le 6$, $u_{6,-3}\le 3$, $u_{6,-5}=1$, $P_6(X)\le 28$. By the same argument as in {\bf Case IIb}, we have $K_X^3\ge\frac{10}{21}$.
\medskip

{\bf Case IIIc.} $u_{6,0}=8$, $\mathrm{dim}\psi_{6,0}(U_{6,0})\le 3$, $M_6|_F\ge C+C_{-1}$ with $(C_{-1}\cdot C)\le 3$ ($\Rightarrow K_X^3\ge\frac{5}{12}$).

Since $h^0(F,C_{-1})\geq 5$, we have $C_{-1}\ge C+C'$, where $C'$ is a moving curve on $F$ satisfying $h^0(F, C')\ge 3$. If $|C'|$ and $|C|$ are composed of the same pencil, we have $\beta\ge\frac{2}{3}$. Since $\alpha(7)>2$, we have $\xi\ge\frac{5}{7}$. We get $\alpha(5)\ge(5-1-1-\frac{1}{\beta})\cdot \xi\ge\frac{15}{14}>1$. Thus $\xi\ge\frac{4}{5}$. By our assumption, we have $\xi>\frac{4}{5}$, which gives $\alpha(6)>2$. Then $\varphi_{6,X}$ is birational, which contradicts to our assumption. So $|C'|$ and $|C|$ are not composed of the same pencil.

Therefore we have $M_6|_F\ge 2C+C'$. By \cite[Proposition 3.7]{CHP} (iii), we have $\xi\ge\frac{3}{4}$ ($m_1=6$, $j=2$, $\delta_1=2$) by successfully taking $n=6$, $7$. Thus we have $$(\pi^*(K_X)|_F)^2\ge\frac{2\xi+1}{6}\ge \frac{5}{12}.$$ In particular, we have $K_X^3\ge\frac{5}{12}$.
\medskip

{\bf Case IVa.} $u_{6,-1}=7$, $\mathrm{dim}\psi_{6,0}(U_{6,-1})\le 3$. ($\Rightarrow K_X^3>0.4714$)

When $(S_{6,-1}|_F\cdot C)\le 3$, we have $S_{6,-1}|_F\ge 2C+C_{-2}$, where $C_{-2}$ is a moving curve on $F$ satisfying $h^0(F, C_{-2})\ge 3$. By the same argument as in the last part of {\bf Case 3} of Proposition \ref{prop:P6}, $\varphi_{6,X}$ is birational, a contradiction.

Thus we only need to consider the case when $(S_{6,-1}|_F\cdot C)\ge 4$.
We have $S_{6,-1}|_F\ge C+C_{-1}$, where $C_{-1}$ is a moving curve on $F$ satisfying $h^0(F,C_{-1})\ge 4$.

If $(C_{-1}\cdot C)\ge 4$, $\varphi_{6,X}$ is birational by Theorem \ref{k3} (i.2) ($m_1=6$, $j_1=j_2=1$, $\delta_2=4$, $\mu=1$, $\beta=\frac{1}{2}$), which is a contradiction.

Thus $(C_{-1}\cdot C)\le 3$. So we have $C_{-1}\ge C+C'$, where $C'$ is a moving curve on $F$. When $|C'|$ and $|C|$ are not composed of the same pencil, we have $(C'\cdot C)\ge 2$. By Theorem \ref{k3} (ii.1), we have $\xi\ge\frac{5}{7}$ ($n=6$, $m_1=6$, $j_1=1$, $j_2=2$, $\mu=1$, $\xi=\frac{2}{3}$, $\delta_2=2$). Theorem \ref{k3} (ii.2) implies that $\varphi_{6,X}$ is birational ($m_1=6$, $j_1=1$, $j_2=2$, $\delta_2=2$, $\mu=1$, $\xi=\frac{5}{7}$), which is a contradiction.

So $C'\sim C$ and we have $S_{6,-1}|_F\ge 3C$. By  \cite[Proposition 3.5]{CHP}, we have $\beta\ge\frac{4}{7}$. Since $\alpha(7)>2$, we have $\xi\geq \frac{5}{7}$. Since
$\alpha(8)\ge(8-1-1-\frac{1}{\beta})\cdot \xi>3$, we have $\xi\ge\frac{3}{4}$.
Denote by $\xi_{6,-1}$ the intersection number $(\pi^*(K_X)|_F\cdot S_{6,-1}|_F)$.
We have $\xi_{6,-1}\ge \beta (C\cdot S_{6,-1}|_F)\ge \frac{16}{7}$.  Besides,
Kawamata-Viehweg vanishing theorem implies
$$|K_{X'}+M_6||_F\lsgeq |K_F+S_{6,-1}|_F|\lsgeq |C+S_{6,-1}|_F|,$$
which directly implies $7\pi^*(K_X)|_F\ge C+S_{6,-1}|_F$ since $|C+S_{6,-1}|_F|$ is base point free.  Noting that a general $S_{6,-1}|_F$ can be smooth, nef and big, we may use the
the similar method to that of \cite[Proposition 3.6]{CHP} (2.1) to obtain the following inequality, for
any $n\ge 8$,
$$(n+1)\xi_{6,-1}\ge \ulcorner (n-6)\xi_{6,-1}\urcorner+16,$$
where one notes that $((K_F+S_{6,-1}|_F)\cdot S_{6,-1}|_F)\ge 16$.  Take $n=8$, we get $\xi_{6,-1}\ge\frac{7}{3}$. Take $n=10$, we get $\xi_{6,-1}\ge\frac{26}{11}$. Take $n=9$, we get $\xi_{6,-1}\ge\frac{12}{5}$.  Since $7\pi^*(K_X)|_F\ge C+S_{6,-1}|_F$ and $\xi\ge\frac{3}{4}$, we have $(\pi^*(K_X)|_F)^2\ge\frac{63}{140}$. By \eqref{eq:estimation1.2}, we have $K_X^3>0.4714$. 
\medskip

{\bf Case IVb.} $u_{6,-1}=6$ and $\mathrm{dim}\psi_{6,0}(U_{6,-1})\le 3$. ($\Rightarrow K_X^3\ge\frac{11}{28}$)

If $(S_{6,-1}|_F\cdot C)\ge 4$, we have $\xi\ge\frac{3}{4}$ by \cite[Proposition 3.6]{CHP} (1.1) ($n=11$, $\delta=4$, $m_1=6$, $\beta=\frac{1}{2}$, $\xi=\frac{2}{3}$). By the similar reason to that in Case IVa, we have $7\pi^*(K_X)|_F\ge C+S_{6,-1}|_F$. Thus we get $K_X^3\ge(\pi^*(K_X)|_F)^2\ge\frac{11}{28}$.

If $(S_{6,-1}|_F\cdot C)\le 3$, we have $S_{6,-1}|_F\ge 2C+C_{-2}$, where $C_{-2}$ is a moving curve on $F$. When $|C_{-2}|$ and $|C|$ are composed of the same pencil, we have $\beta\ge\frac{4}{7}$ by  \cite[Proposition 3.5]{CHP}.  Besides, $\alpha(7)>2$ implies $\xi\geq \frac{5}{7}$. 
Thus we have $(\pi^*(K_X)|_F)^2\ge\frac{20}{49}$. When  $|C_{-2}|$ and $|C|$ are not composed of the same pencil, we have $(C_{-2}\cdot C)\ge 2$. By Theorem \ref{k3} (ii.1), we have $\xi\ge\frac{5}{7}$ ($n=6$, $m_1=6$, $j_1=1$, $j_2=2$, $\mu=1$, $\xi=\frac{2}{3}$, $\delta_2=2$). Theorem \ref{k3} (ii.2) implies that $\varphi_{6,X}$ is birational ($n=6$, $m_1=6$, $j_1=1$, $j_2=2$, $\mu=1$, $\xi=\frac{5}{7}$, $\delta_2=2$), which contradicts to our assumption.
\medskip

{\bf Case V.} $u_{6,-2}=6$. ($\Rightarrow K_X^3>0.4771$)

If $(S_{6,-2}|_F\cdot C)\le 3$, we have $S_{6,-2}|_F\ge 2C+C_{-2}$, where $C_{-2}$ is a moving curve on $F$. By the same argument as in the last part of {\bf Case 4} of Proposition \ref{prop:P6}, we conclude that $\varphi_{6,X}$ is birational, a contradiction.

So we have $(S_{6,-2}\cdot C)\ge 4$ in this case.  In fact, the case $(S_{6,-2}\cdot C)\ge 5$ has been treated in {\bf Case 4} of Proposition \ref{prop:P6}, which shows that $\varphi_{6,X}$ is birational (a contradiction). Hence $(S_{6,-2}|_F\cdot C)=4$. Theorem \ref{k2} (i) implies that $\varphi_{6,X}$ is birational if $\beta>\frac{1}{2}$. Thus we have $\beta=\frac{1}{2}$. By Theorem \ref{k2} (iii) and (iv), we have $\xi\geq \frac{3}{4}$ and, for any $n\ge 6$,  $$(n+1)\xi\ge\ulcorner(n-5)\xi \urcorner+4.$$  Take $n=8$, we get $\xi\ge\frac{7}{9}$.
Take $n=9$, we get $\xi\ge\frac{4}{5}$. By our assumption, we have $\xi>\frac{4}{5}$. Take $n=10$ in the above inequality, we have $\xi\ge\frac{9}{11}$. By Theorem \ref{k2} (iii) ($m_1=6$, $j=2$, $(S_{6,-2}|_F\cdot C)=4$), we have $(\pi^*(K_X)|_F)^2\ge\frac{5}{11}$. By \eqref{eq:estimation1.2}, we have $K_X^3>0.4771$.
\medskip

{\bf Case VI.} $u_{6,-3}=4$. ($\Rightarrow K_X^3>0.4734$)

If $(S_{6,-3}|_F\cdot C)\ge 4$, $\varphi_{6,X}$ is birational by Theorem \ref{k2} (i) ($m_1=6$, $j=3$, $\tilde{\delta}=4$, $\beta=\frac{1}{2}$, $\xi=\frac{5}{7}$), which contradicts to our assumption.
Thus we have $(S_{6,-3}|_F\cdot C)\le 3$. So $S_{6,-3}|_F\ge C+C_{-1}$, where $C_{-1}$ is a moving curve on $F$.

If $|C_{-1}|$ and $|C|$ are not composed of the same pencil, we have $(C_{-1}\cdot C)\ge 2$. By Theorem \ref{k3} (i.1), we have $\xi\ge\frac{4}{5}$ ($n=4$, $m_1=6$, $j_1=3$, $j_2=1$, $\beta=\frac{1}{2}$, $\delta_2=2$). Our assumption implies that we have $\xi>\frac{4}{5}$. Theorem \ref{k3} (i.2) implies that $\varphi_{6,X}$ is birational ($m_1=6$, $j_1=3$, $j_2=1$, $\xi>\frac{4}{5}$, $\beta=\frac{1}{2}$), which is a contradiction.

Thus $|C_{-1}|$ and $|C|$ are composed of the same pencil.  So $S_{6,-3}|_F\ge 2C$. By  \cite[Proposition 3.5]{CHP}, we have $\beta\ge\frac{5}{9}$. By Theorem \ref{kf} (i.1), we have $\xi\ge\frac{4}{5}$ ($n=4$, $m_1=6$, $j_1=3$, $j_2=2$, $\beta=\frac{5}{9}$). So $(\pi^*(K_X)|_F)^2\ge\frac{4}{9}$. By \eqref{eq:estimation1.2}, we have $K_X^3>0.4734$.
\medskip

{\bf Case VII.} $u_{6,-5}=2$. ($\Rightarrow K_X^3>0.4362$)

If $|S_{6,-5}|_F|$ and $|C|$ are not composed of the same pencil, we have $(S_{6,-5}|_F\cdot C)\ge 2$. By Theorem \ref{k2} (iv), we get $\xi\geq \frac{5}{7}$ by taking $n=6$ and $\xi\geq \frac{3}{4}$ by taking $n=7$.  By Theorem \ref{k2} (iii), we have $(\pi^*(K_X)|_F)^2\ge\frac{19}{44}$. By \eqref{eq:estimation1.2}, we have $K_X^3>0.4746$.

If $|S_{6,-5}|_F|$ and $|C|$ are composed of the same pencil, we have $\beta\ge\frac{6}{11}$. By Theorem \ref{kf}(i.1), we get $\xi\geq \frac{5}{7}$ by taking $n=6$ and $\xi\geq \frac{3}{4}$ by taking $n=7$. We have $(\pi^*(K_X)|_F)^2\ge\frac{9}{22}$. By \eqref{eq:estimation1.2}, we obtain $K_X^3>0.4362$.
\medskip

Now we prove (2). By the results of {\bf Case I} $\sim$ {\bf Case VII},  we only need to consider the case where $u_{6,0}\le 7$, $u_{6,-1}\le 5$, $u_{6,-3}\le 3$, $u_{6,-5}=1$. We have $\mu\ge\frac{7}{6}$ and $\beta\ge\frac{7}{13}$.  As $\alpha(7)>2$, we have $\xi\geq \frac{5}{7}$. So we get $K_X^3\ge\frac{35}{78}>\frac{11}{28}$.

For (3), by the results of {\bf Case I} $\sim$ {\bf Case VII}, we are left to treat the following cases:
\begin{enumerate}
\item[(3.1)] $u_{6,0}=8$, $\xi\ge\frac{3}{4}$, $(\pi^*(K_X)|_F)^2\ge\frac{5}{12}$, $u_{6,-4}=3$;
\item[(3.2)] $u_{6,0}=8$, $\xi\ge\frac{3}{4}$, $(\pi^*(K_X)|_F)^2\ge\frac{5}{12}$, $u_{6,-1}\le 6$, $u_{6,-2}\le 5$, $u_{6,-3}\le 3$, $u_{6,-4}\le 2$, $u_{6,-5}=1$;
\item[(3.3)] $u_{6,0}\le 7$, $u_{6,-1}\le 6$, $u_{6,-2}\le 5$, $u_{6,-3}\le 3$, $u_{6,-4}\le 2$, $u_{6,-5}=1$;
\item[(3.4)] $u_{6,0}\le 7$, $u_{6,-1}\le 6$, $u_{6,-2}\le 5$, $u_{6,-3}\le 3$, $u_{6,-4}= 3$, $u_{6,-5}=1$.
\end{enumerate}
We first treat (3.1). If $|S_{6,-4}|_F|$ and $|C|$ are composed of the same pencil, we have $\beta\ge\frac{3}{5}$. We get $K_X^3\ge\beta\xi\geq\frac{9}{20}>0.4328$. If $|S_{6,-4}|_F|$ and $|C|$ are not composed of the same pencil, we have $(\pi^*(K_X)|_F\cdot S_{6,-4}|_F)\ge 1$.  We have $K_X^3>0.4328$ by \eqref{eq:estimation1.2}.
Next we treat $(3.2)$. We have $\mu\ge\frac{7}{6}$. By Inequality \eqref{cri}, we have $\beta\ge\frac{7}{13}$. Then we get $K_X^3\ge\frac{49}{104}>0.4328$.
For (3.3), we have $\mu\ge\frac{4}{3}$ and $\beta\ge\frac{4}{7}$. Since $\alpha(5)>1$, we get $\xi\ge\frac{4}{5}$. Our assumption implies that we have $\xi>\frac{4}{5}$. Thus we have $\alpha(6)>2$, which implies that $\varphi_{6,X}$ is birational, which is a contradiction. So $(3.3)$ does not occur. Finally, for $(3.4)$, we have $\mu\geq \frac{7}{6}$ and $\beta\geq \frac{7}{13}$. Since we have $\xi\geq \frac{5}{7}$, so we get $K_X^3\geq \frac{35}{78}>0.4328$. 

Now we consider $(4)$. By the arguments in {\bf Case I}-{\bf Case VII}, we only need to treat the following cases:

\begin{enumerate}
\item[(4.1)] $u_{6,0}=8$, $u_{6,-1}\le 6$, $u_{6,-2}=5$, $u_{6,-3}\le 3$, $u_{6,-5}=2$;
\item[(4.2)] $u_{6,0}\le 8$, $u_{6,-1}\le 6$, $u_{6,-2}\le 4$, $u_{6,-3}\le 3$, $u_{6,-5}=2$;%$u_{6,-6}=1$;
\item[(4.3)] $u_{6,0}\le 8$, $u_{6,-1}\le 6$, $u_{6,-2}=5$, $u_{6,-3}\le 3$, $u_{6,-5}=1$;
\item[(4.4)] $u_{6,0}\le 8$, $u_{6,-1}\le 6$, $u_{6,-2}\le 4$, $u_{6,-3}\le 3$, $u_{6,-5}=1$;
\item[(4.5)] $u_{6,0}\leq 7$, $u_{6,-1}\le 6$, $u_{6,-2}=5$, $u_{6,-3}\le 3$, $u_{6,-5}=2$.
\end{enumerate}

(4.1) By the argument in {\bf Case VII}, we only need to treat the case when $|S_{6,-5}|_F|$ and $|C|$ are composed of the same pencil. In particular, we have $\beta\ge\frac{6}{11}$.

We claim that $(S_{6,-2}|_F\cdot C)\le 3$. Otherwise,  Theorem \ref{k2} (i) and (ii) imply that $\varphi_{6,X}$ is birational ($m_1=6$, $j=2$, $\tilde{\delta}\ge 4$, $\xi\ge\frac{2}{3}$, $\beta\ge\frac{6}{11}$), which contradicts to our assumption. Thus we have $S_{6,-2}|_F\ge C+C_{-1}$, where $C_{-1}$ is a moving curve on $F$ satisfying $h^0(F,C_{-1})\ge 3$.

If $|C_{-1}|$ and $|C|$ are composed of the same pencil, we have $S_{6,-2}|_F\ge 3C$. By Theorem \ref{kf} (ii.1), we have $\xi\ge\frac{4}{5}$ ($n=4$). Theorem \ref{kf} (ii.2) implies that $\varphi_{6,X}$ is birational ($m_1=6$, $j_1=2$, $j_2=3$, $\xi=\frac{4}{5}$, $\mu=1$), which contradicts to our assumption. Thus $|C_{-1}|$ and $|C|$ are not composed of the same pencil. In particular, we have $(C_{-1}\cdot C)\ge 2$. By Theorem \ref{k3} (i.1), we have  $\xi\ge\frac{4}{5}$ ($n=4$, $m_1=6$, $j_1=2$, $j_2=1$, $\beta=\frac{6}{11}$, $\xi=\frac{5}{7}$, $\delta_2=2$) .

By the arguments in {\bf Case IIIa}-{\bf Case IIIc,} %and Lemma \ref{lem:restric vol},
we have $M_6|_F\ge 2C+C'$, where $C'$ is a moving curve on $F$ satisfying $h^0(F, C')\ge 3$. Moreover, $|C'|$ and $|C|$ are not composed of the same pencil. Thus we have $$(\pi^*(K_X)|_F)^2\ge\frac{2\xi+\beta(C\cdot C')}{6}\ge \frac{74}{165}.$$
By \eqref{eq:estimation1.2} ($m=6$, $l=5$), we have $K_X^3>0.4766$. Thus $(4)$ holds under the assumption of $(4.1)$.

 $(4.2)$. By the argument in {\bf Case VII}, we only need to consider the case when $|S_{6,-5}|_F|$ and $|C|$ are composed of the same pencil. By  \cite[Proposition 3.5]{CHP}, we have $\beta\ge\frac{6}{11}$. Since $P_6(X)\ge 28$, we have $\mu\ge\frac{7}{6}$ by our assumption.  Note that we have $\xi\geq \frac{5}{7}$. Since $\alpha(8)>3$, we get
 $\xi\geq \frac{3}{4}$. 
  Thus we have $K_X^3\ge \mu\beta\xi >0.4772$. So $(4)$ holds.

$(4.3)$ \& $(4.5)$. We have $\mu\ge\frac{7}{6}$ by our assumption. By Inequality \eqref{cri}, we have $\beta\ge\frac{7}{13}$.  Since $u_{6,-2}=5$, by the same argument as in $(4.1)$, we get $\xi\ge\frac{4}{5}$.  Hence $K_X^3\ge \mu\beta\xi\ge\frac{98}{195}>\frac{1}{2}$.  So $(4)$ holds.

 $(4.4)$. We have $\mu\ge\frac{4}{3}$ by our assumption. By Inequality \eqref{cri}, we have $\beta\ge\frac{4}{7}$. Since $\alpha(5)>1,$ we have $\xi\ge\frac{4}{5}$. By our assumption, we have $\xi>\frac{4}{5}$. Thus $\alpha(6)>2$, which implies that $\varphi_{6,X}$ is birational, which contradicts to our assumption. So $(4.4)$ does not occur.

 For $(5)$,  by the arguments in Case V ($\Rrightarrow \beta=\frac{1}{2}$) and Case VI ($\Rrightarrow \beta\geq \frac{5}{9}$), we see that $u_{6,-2}=6$ and $u_{6,-3}=4$ can not hold simultaneously.  Combining all the arguments in {\bf Case I}-{\bf Case VII}, we only need to consider the following cases:

 \begin{enumerate}
 \item[(5.1)] $u_{6,0}\le 8$, $u_{6,-1}=7$, $\beta\ge\frac{4}{7}$, $u_{6,-2}\le 5$, $u_{6,-3}=4$, $u_{6,-4}\le 3$, $u_{6,-5}\le 2$, $u_{6,-6}=1$;

 \item[(5.2)] $u_{6,0}\le 8$, $u_{6,-1}=7$, $\beta\ge\frac{4}{7}$, $u_{6,-2}\le 5$, $u_{6,-3}\le 3$, $u_{6,-4}\le 3$, $u_{6,-5}\le 2$, $u_{6,-6}=1$;

\item[(5.3)]  $u_{6,0}\le 8$, $u_{6,-1}\leq 6$, $u_{6,-2}=6$, $u_{6,-3}\leq 3$, $u_{6,-4}\le 3$, $u_{6,-5}\le 2$, $u_{6,-6}=1$;

\item[(5.4)]  $u_{6,0}\le 8$, $u_{6,-1}\leq 6$, $u_{6,-2}\le 5$, $u_{6,-3}=4$, $u_{6,-4}\le 3$, $u_{6,-5}\le 2$, $u_{6,-6}=1$.
\end{enumerate}

$(5.1)$.  By the argument in {\bf Case VI}, we have $\xi\ge\frac{4}{5}$. By our assumption, we have $\mu\ge\frac{7}{6}$. Thus we have $$K_X^3\ge \mu\beta\xi\ge\frac{8}{15}.$$

$(5.2)$. By the assumption, we have $\mu\ge\frac{4}{3}$. Similar to the case $(4.4)$, we see that $\alpha(5)>1$ and $\alpha(6)>2$, which implies the birationality of $\varphi_{6,X}$ (a contradiction).

$(5.3)$ \& $(5.4)$.  Similar to the case $(5.2)$, one has $\mu\ge\frac{4}{3}$, which gives a contradiction.
\end{proof}

\subsection{The classification of ${\mathbb B}^{(5)}(X)$}

\begin{lem}\label{lem:Pm}  Let $X$ be a minimal projective 3-fold of general type with $p_g(X)=2$, $d_1=1$, $\Gamma\cong \bP^1$. Assume that $F$ is a $(1,2)$-surface. Then
\begin{itemize}
\item[(i)] $P_3(X)\geq P_2(X)+2$;
\item[(ii)] $P_4(X)\geq P_3(X)+4$;
\item[(iii)] $P_6(X)\geq P_5(X)+7$.
\end{itemize}
\end{lem}
\begin{proof}
Since $|3K_{X'}||_F\lsgeq |K_{X'}+F_1+F|$ for two distinct general fibers of $f$, it follows from \cite[Lemma 4.6]{MCNoe} that
\begin{align*}
|M_3||_F\lsgeq |C|,
\end{align*}
which implies that $P_3(X)\ge P_2(X)+2$.

Since $|4K_{X'}|\lsgeq |2(K_{X'}+F)|$ for a general fiber of $f$, by Theorem \ref{KaE}, we have
\begin{align*}
|M_4||_F\lsgeq |2\sigma^*(K_{F_0})|,
\end{align*}
which implies that $P_4(X)\ge P_3(X)+4$.

By the same argument as above, we get $|M_6||_F\lsgeq \mathrm{Mov}|\sigma^*(3K_{F_0})|$. $\alpha(6)>1$ implies that $(M_6\cdot C)\ge 4$. Thus  $|M_6||_F\neq \mathrm{Mov}|\sigma^*(3K_{F_0})|$, which implies that we have $P_6(X)\ge P_5(X)+7$.
\end{proof}

Collecting all above results of this section, we want to classify those basket ${\mathbb B}(X)$
which satisfy the following properties:
\begin{enumerate}
\item $3\le P_2(X)\le 4$;
\item $P_2(X)+2\le P_3(X)\le 8$;
\item $P_3(X)+4\le P_4(X)\le 13$;
\item $P_4(X)+4\le P_5(X)\le 21$;
\item $P_5(X)+7\le P_6(X)\le 31$;
\item $\chi(\mathcal{O}_X)=0, -1;$
\item $K_{X}^3\ge\frac{5}{14}$;
\item If $P_6(X)\ge 26$, then $K_X^3\ge \frac{11}{28}$;
\item If $P_6(X)\ge 27$, we have $K_X^3>0.4328$;
\item If $P_6(X)\ge 28$, we have $K_X^3>0.4714$;
\item If $P_6(X)\ge 31$, we have $K_X^3\ge\frac{8}{15}$.
\end{enumerate}
\medskip

This situation naturally fits into the hypothesis of \cite[(3.8)]{EXP1} from which we can
list all the possibilities for $B^{(5)}(X)$.  To be precise,
{\footnotesize
$$B^{(5)}=\{ n^5_{1,2} \times (1,2),n^5_{2,5} \times (2,5), n^5_{1,3}
\times (1,3), n^5_{1,4}  \times (1,4), n^5_{1,5} \times
(1,5),\cdots\}$$} with
$$B^{(5)}\left\{
 \begin{array}{l}
  n^5_{1,2}=3 \chi(\OO_X) +6P_2- 3 P_3 + P_4 - 2 P_5 + P_6+ \sigma_5 ,\\
 n^5_{2,5}= 2 \chi(\OO_X)-P_3 +   2 P_5  - P_6- \sigma_5 \\
 n^5_{1,3}= 2 \chi(\OO_X)+2P_2+ 3 P_3- 3 P_4 -P_5 + P_6   + \sigma_5 , \\
 n^5_{1,4}= \chi(\OO_X) -3P_2+  P_3 +2 P_4 -P_5-\sigma_5 \\
  n^5_{1,r}=n^0_{1,r}, r \geq 5
\end{array} \right. $$
where $\sigma_5=\sum_{r\geq 5}n_{1,r}^0\geq 0$ and
$$\sigma_5\leq 2\chi(\OO_X)-P_3+2P_5-P_6.$$
Note also that, by our definition, each of the above coefficients satisfies $n_{*,*}^0\geq 0$.

Inputing above constraints, our independently written  computer programs output a raw list for $\{B_X^{(5)}, P_2(X),\chi(\OO_X)\}$. Taking into account those possible packings, we finally get the list    ${\mathbb S}_2$   which consists of 263 elements. Being aware of the length of this paper, we do not list the set ${\mathbb S}_2$, which can be found, however, at
	\begin{center}
		\verb|http://www.dima.unige.it/~penegini/publ.html|
	\end{center}   

\subsection{Proof of Theorem \ref{thm:main p_g2}}
\begin{proof} Theorem \ref{thm:main p_g2} follows directly from Lemma \ref{lem:phi6 general case}, Proposition \ref{prop:(2,3) surface}, Remark \ref{rem:xi=2/3}, Lemma \ref{lem:charac K1},  Proposition \ref{prop:P2}, Proposition \ref{prop:P3}, Proposition \ref{prop:P4}, Proposition \ref{prop:P5}, Proposition \ref{prop:P6} and Proposition \ref{prop:esti of volume}.
\end{proof}

%%%%%%%% References %%%%%%%%%%

\end{document}